\documentclass[11pt]{amsart}
\usepackage{amsmath, amssymb, verbatim,color}
\newcommand{\de}{\partial}
\newcommand{\db}{\overline{\partial}}

\newcommand{\ddt}{\frac{\partial}{\partial t}}

\newcommand{\ddbar}{\sqrt{-1} \partial \overline{\partial}}
\newcommand{\Ric}{\mathrm{Ric}}
\newcommand{\ov}[1]{\overline{#1}}
\newcommand{\mn}{\sqrt{-1}}
\newcommand{\tr}[2]{\mathrm{tr}_{#1}{#2}}
\newcommand{\ti}[1]{\tilde{#1}}
\newcommand{\vp}{\varphi}

\newcommand{\ve}{\varepsilon}
\newcommand{\of}{\omega_{\textrm{SRF}}}
\newcommand{\F}{\mathcal{F}}

\renewcommand{\leq}{\leqslant}
\renewcommand{\geq}{\geqslant}
\renewcommand{\le}{\leqslant}
\renewcommand{\ge}{\geqslant}

\numberwithin{equation}{section}

\begin{document}
\newtheorem{claim}{Claim}
\newtheorem{theorem}{Theorem}[section]
\newtheorem{lemma}[theorem]{Lemma}
\newtheorem{corollary}[theorem]{Corollary}
\newtheorem{proposition}[theorem]{Proposition}
\newtheorem{question}[theorem]{Question}
\newtheorem{conjecture}[theorem]{Conjecture}

\theoremstyle{definition}
\newtheorem{remark}[theorem]{Remark}

\title[Collapsing limits]{The K\"ahler-Ricci flow, Ricci-flat metrics and collapsing limits}
\author[V. Tosatti]{Valentino Tosatti}
\thanks{Research supported in part by NSF grants DMS-1308988, DMS-1332196 and DMS-1406164.  The first-named author is supported in part by a Sloan Research Fellowship.}
\address{Department of Mathematics, Northwestern University, 2033 Sheridan Road, Evanston, IL 60208}
\author[B. Weinkove]{Ben Weinkove}
\author[X. Yang]{Xiaokui Yang}
\begin{abstract}  We investigate the K\"ahler-Ricci flow on holomorphic fiber spaces whose generic fiber is a Calabi-Yau manifold.  We establish uniform metric convergence to a metric on the base, away from the singular fibers, and show that the rescaled metrics on the fibers converge to Ricci-flat K\"ahler metrics.  This strengthens previous work of Song-Tian and others.  We obtain analogous results for degenerations of Ricci-flat K\"ahler metrics.
\end{abstract}

\maketitle

\section{Introduction}

\subsection{Background}  This paper establishes convergence results for collapsing metrics along the K\"ahler-Ricci flow and along families of Ricci-flat K\"ahler metrics.  We give now some background and motivation for these results,
glossing over some technicalities for the moment.

Let $X$ be a compact K\"ahler manifold, and suppose we have a holomorphic surjective map $\pi: X \rightarrow B$ onto another compact K\"ahler manifold $B$, with $0<\dim B < \dim X.$
We consider a family of K\"ahler metrics $(\omega(t))_{t\in [0,\infty)}$ on  $X$ with the property that the cohomology classes $[\omega(t)] \in H^{1,1}(X, \mathbb{R})$ have the limiting behavior:
$$[\omega(t)] \rightarrow \pi^*\alpha, \quad \textrm{in } H^{1,1}(X, \mathbb{R}),$$
where $\alpha$ is a K\"ahler class on $B$.  Moreover, we suppose that the K\"ahler metrics $\omega(t)$ solve one of the two natural PDEs:  the K\"ahler-Ricci flow or the Calabi-Yau equation.
Note that from topological considerations, the metrics $\omega(t)$ are automatically \emph{volume collapsing}, meaning that
$$\int_X \omega(t)^{\dim X} = [\omega(t)]^{\dim X} \rightarrow 0, \quad \textrm{as } t\rightarrow \infty,$$
since $B$ has dimension strictly less than $\dim X$.

A natural question is: does $(X, \omega(t))$ converge to  $(B,\omega_B)$, where $\omega_B$ is a unique canonical metric in the class $\alpha$?

 In the cases we consider, the existence of the ``canonical metric'' has already been established, and the remaining point is to  establish convergence in the strongest possible sense.
There are weak notions of convergence, such as convergence of $\omega(t)$ to $\omega_B$ as currents, or convergence at the level of K\"ahler potentials in $C^{1,\alpha}$, say,  for $\alpha \in (0,1)$.  A stronger notion of convergence is convergence of $\omega(t)$ to $\omega_B$  in the $C^0$ norm.  This kind of convergence implies global Gromov-Hausdorff convergence.  The main point of this paper is that we can obtain, in many cases, convergence of the stronger type when before only weaker convergence was known.

To be more explicit, we discuss now the simplest case of such collapsing for the K\"ahler-Ricci flow, which is already completely understood.  Let $X$ be a product $X= E \times B$ where $B$ is a compact Riemann surface of genus larger than $1$ and $E$ is a one-dimensional torus (an elliptic curve).  We consider the K\"ahler-Ricci flow \cite{Ha, Ca}
\begin{equation} \label{firstkrf}
\ddt{} \omega = - \Ric(\omega) - \omega, \quad \omega|_{t=0}=\omega_0,
\end{equation}
starting at an arbitrary initial K\"ahler metric $\omega_0$ on $X$.    A solution to (\ref{firstkrf}) exists for all time \cite{TZ}.
Note that (\ref{firstkrf}) is really the \emph{normalized} K\"ahler-Ricci flow.  It has the property that as $t\rightarrow \infty$ the K\"ahler classes $[\omega(t)] \in H^{1,1}(X, \mathbb{R})$ converge to the class $\pi^* \alpha$ where $\alpha = [\omega_B]$, for $\omega_B$ the unique K\"ahler-Einstein metric on $B$ satisfying
\begin{equation} \label{uke}
\Ric(\omega_B)= - \omega_B.
\end{equation}
Here $\pi:X \rightarrow B$ is the projection map.

This example was first studied by Song-Tian \cite{ST} who considered the much more general situation of a holomorphic fiber space $\pi: X \rightarrow B$ of Kodaira dimension one, whose generic fibers are elliptic curves.  For the product $X= E \times B$, the results of \cite{ST, SW, Gi, FZ} show that one obtains $C^{\infty}$ convergence of $\omega(t)$ to $\omega_B$.  Moreover, the  curvature of $\omega(t)$ and its covariant derivatives  remain bounded as $t\rightarrow \infty$.  Geometrically, the metrics shrink in the directions of $E$ and we obtain convergence to the K\"ahler-Einstein metric $\omega_B$ on $B$.  Moreover, if we look at the evolving metrics restricted to any given fiber $E$ and rescale appropriately,  they converge to a flat metric on $E$.

A natural generalization of this result is to consider $X=E \times B$ where now $E$ is a Calabi-Yau manifold (i.e. $c_1(E)=0$ in $H^{1,1}(E,\mathbb{R})$) and $B$ is a compact K\"ahler manifold with $c_1(B)<0$.  The behavior of the K\"ahler-Ricci flow on the manifolds $E$ and $B$ separately is well-known by the results of Cao \cite{Ca}:  on $E$ the unnormalized K\"ahler-Ricci flow $\ddt{} \omega= - \Ric(\omega)$ converges to a K\"ahler Ricci-flat metric, and on $B$ the normalized K\"ahler-Ricci flow (\ref{firstkrf}) converges to the unique K\"ahler-Einstein metric satisfying (\ref{uke}).   However, the behavior of the flow on the product $X=E\times B$ is in general harder to understand than the two-dimensional example above.  A source of difficulty is that, except in dimension one, Calabi-Yau manifolds may not admit flat metrics.

In the case when $E$ \emph{does} admit a flat K\"ahler metric, it was shown by Gill \cite{Gi}, Fong-Zhang \cite{FZ} that  the same results hold as in the case of a product of Riemann surfaces.  In particular, the curvature remains bounded along the flow.  These methods fail for a general $E$.  Indeed,
by starting the flow at a product of a non-flat Ricci-flat metric on $E$ with a K\"ahler-Einstein metric on $B$, it is easy to check that the curvature along  the K\"ahler-Ricci flow (\ref{firstkrf}) must blow up.  One can still obtain ``weak'' convergence of the metrics \cite{ST2, FZ}.  Indeed,  the evolving metrics along the flow are comparable to the ``model'' collapsing metrics, which implies in particular that the fibers shrink and one obtains convergence to $\omega_B$ at the level of potentials in the $C^{1,\alpha}$ norm.  However, until now the question of whether the metrics $\omega(t)$ converge (in $C^0$ say) to $\omega_B$ has remained open.

Our results show that in fact  the metrics $\omega(t)$ converge exponentially fast in the $C^0$ norm to $\omega_B$, which implies in particular Gromov-Hausdorff convergence of $(X, \omega(t))$ to $(B, \omega_B)$.  Moreover, we show that if one restricts to a fiber $E$ over $y\in B$ then the metrics $\omega(t)$, appropriately rescaled, converge exponentially fast to a Ricci-flat K\"ahler metric on $E$ and this convergence is uniform in $y$.
All of these statements are new in this setting, and require new estimates, which we explain later.

Of course, the assumption that $X$ is a product is rather restrictive.  In fact, the same results hold when $\pi: X \rightarrow B$ is a holomorphic submersion, except that the K\"ahler-Einstein equation (\ref{uke}) is replaced by the twisted K\"ahler-Einstein equation \cite{ST,ST2, Fi}.  But
even this is not sufficiently general, since we also need to allow $\pi:X \rightarrow B$ to admit singular fibers, and for $B$ itself to be possibly singular. More precisely, we assume that
$\pi:X\to B$ is a fiber space, namely a surjective holomorphic map with connected fibers, and $B$ is an irreducible normal projective variety. In this case we obtain the same metric convergence results on compact subsets away from the singular fibers (but this time not exponential).  Again it was only known before that convergence occurs at the level of potentials \cite{ST2}.

The latter situation is really quite general.  Indeed, if one believes the Abundance Conjecture from algebraic geometry (that on a compact K\"ahler manifold the canonical bundle $K_X$ is nef implies $K_X$ is semiample) then whenever the K\"ahler-Ricci flow has a long-time solution, there exists such a fiber space map $\pi: X \rightarrow B$.  For the above discussion to apply, we also need to remove the two extreme cases:  when $B$ has dimension $0$ or $\dim X$.
In the first case $X$ is Calabi-Yau and it is known by the work of Cao \cite{Ca} that the K\"ahler-Ricci flow (unnormalized) converges smoothly to Yau's  Ricci-flat K\"ahler metric \cite{Ya}. In the second case $X$ is of general type (in particular, $X$ is projective), and the normalized K\"ahler-Ricci flow converges to a K\"ahler-Einstein current on $X$ \cite{Ts,TZ,Zh}, smoothly away from a subvariety.

On a general projective variety $X$, the K\"ahler-Ricci flow may encounter  finite time singularities.  According to the analytic minimal model program  \cite{ST, Ti2, ST2,ST3} the K\"ahler-Ricci flow will perform a surgery at each finite time singularity, corresponding to a birational operation of the minimal model program from algebraic geometry, and arrive at the  \emph{minimal model} $X_{\textrm{min}}$ with nef canonical bundle (see also \cite{BCHM, EGZ,EGZ2,FZ,LT,  So, SSW, ST4, SW0, SW1, SW2, SY,  TZ, Ts}).  It is then expected that the flow on $X_{\textrm{min}}$ will converge at infinity to a canonical metric on the lower dimensional canonical model $X_{\textrm{can}}$. Our results here are concerned with a special case of this last step, where $X_{\textrm{min}}=X$ is smooth with semiample canonical bundle and $X_{\textrm{can}}$ is what we call $B$.

Finally, we discuss the Ricci-flat K\"ahler metrics.  A similar situation arises here, where one considers a Calabi-Yau manifold $X$ which admits a holomorphic fiber space structure $\pi:X\to B$ as above (possibly with singular fibers and with $B$ singular), and studies the collapsing behavior of the Ricci-flat K\"ahler metrics $\omega(t)$ in the class $[\pi^*\chi]+e^{-t}[\omega_X]$ (where $\omega_X$ is a K\"ahler metric on $X$ and $\chi$ on $B$), as $t$ tends to infinity. Again one expects to have collapsing to a canonical K\"ahler metric $\omega_B$ on $B\backslash S'$ (where $S'$ denotes the singularities of $B$ together with the images of the singular fibers). Furthermore, one would like to understand if a Gromov-Hausdorff limit $(X_\infty,d_\infty)$ of $(X,\omega(t))$ exists (without passing to subsequences), whether it is sufficiently regular, and how it is related to $(B\backslash S',\omega_B)$. It is expected \cite{To3, To4} that $(X_\infty,d_\infty)$ contains an open and dense subset which is a smooth Riemannian manifold, that $(X_\infty,d_\infty)$ is isometric to the metric completion of $(B\backslash S',\omega_B)$, and that $X_\infty$ is homeomorphic to $B$. By comparison, in the case when  the volume of the Ricci-flat metrics is non-collapsing, we have a very thorough understanding of their limiting behavior, thanks to the results in \cite{CT,RZ,RZ2,RZ3,So2,To2}, and all the analogous questions have affirmative answers.

This setup is also related to the Strominger-Yau-Zaslow picture of mirror symmetry \cite{SYZ}. Indeed, Kontsevich-Soibelman \cite{KS}, Gross-Wilson \cite{GW} and Todorov (see \cite[p. 66]{Ma}) formulated a conjecture about certain collapsed limits of Ricci-flat K\"ahler metrics on Calabi-Yau manifolds, where the K\"ahler class is fixed but the complex structure degenerates (to a so-called {\em large complex structure limit}).
This might seem very different from what we are considering here, where the complex structure is fixed and the K\"ahler class degenerates, but in the case when the Calabi-Yau manifolds are actually hyperk\"ahler, one may perform a hyperk\"ahler rotation and often reduce this conjecture to studying the collapsed limits of Ricci-flat metrics in the precise setup of Theorem \ref{mainthm3} (see \cite{GTZ}), although in this case the smooth fibers are always tori. This approach was pioneered by Gross-Wilson \cite{GW}, who solved this conjecture for certain elliptically fibered $K3$ surfaces (see also the discussion after Theorem \ref{mainthm3}).

Going back to our discussion, our new results imply that the Ricci-flat metrics $\omega(t)$ converge in the $C^0$ norm to $\pi^*\omega_B$ on compact sets away from the singular fibers, improving earlier results of the first-named author \cite{To} who proved weak convergence as currents and in the $C^{1,\alpha}$ norm of potentials. Also, the rescaled metrics along the fibers converge to Ricci-flat metrics. Furthermore, we are able to show in Corollary \ref{corgh} that any Gromov-Hausdorff limit contains an open and dense subset which is a smooth Riemannian manifold (this fact in itself is quite interesting and does not follow from the general structure theory of Cheeger-Colding \cite{CC}). These new results solve some questions in \cite{To3, To4}.

\subsection{Metric collapsing along K\"ahler-Ricci flow}

Let $(X,\omega_0)$ be a compact K\"ahler manifold with canonical bundle $K_X$ semiample. By definition this means that there exists $\ell\geq 1$ such that the line bundle $K_X^\ell$ is globally generated.  If we choose $\ell$ sufficiently large and divisible, the sections in $H^0(X,K_X^\ell)$ give a holomorphic map to projective space, with image $B$ an irreducible normal projective variety, and we obtain a surjective holomorphic map $\pi:X\to B$ with connected fibers (see e.g. \cite[Theorem 2.1.27]{La}).  The dimension of $B$ is the \emph{Kodaira dimension of $X$}, written $\kappa(X)$.  We assume $0<\kappa(X)<\dim X$.

We let $m=\kappa(X) = \dim B$ and  $n+m=\dim X$.  The generic fiber $X_y=\pi^{-1}(y)$ of $\pi$ has $K_{X_y}^\ell$ holomorphically trivial, so it is a Calabi-Yau manifold of dimension $n$.

We first consider a special case, where some technicalities disappear and for which our results are stronger.  Assume:
\begin{equation*}
(*) \qquad
\begin{array}{l} \textrm{$\pi: X \rightarrow B$ is a holomorphic submersion onto} \\
\textrm{a compact K\"ahler manifold $B$ with $c_1(B)<0$,}\\
\textrm{with fibers Calabi-Yau manifolds}. \end{array}
\end{equation*}
As we will see in Section \ref{sectionkrfsubmersion}, this assumption implies that the canonical bundle of $X$ is semiample.
In this case, \emph{all} the fibers $X_y = \pi^{-1}(y)$ are  Calabi-Yau manifolds of dimension $n$ (there are no singular fibers).  Song-Tian \cite{ST2} (see also \cite{Fi}) showed that $B$ admits a unique smooth \emph{twisted K\"ahler-Einstein metric} $\omega_B$ solving
\begin{equation} \label{gkee}
\textrm{Ric}(\omega_B) = -\omega_B + \omega_{\textrm{WP}},
\end{equation}
where $\omega_{\textrm{WP}}$ is the Weil-Petersson form on $B$ (see Section \ref{subsectionsub} below).  We will often write $\omega_B$ for $\pi^*\omega_B$.

By Yau's theorem \cite{Ya}, each of the fibers $X_y$ of $\pi:X \rightarrow B$ admits a  unique Ricci-flat K\"ahler metric cohomologous to $\omega_0|_{X_y}$.  Write $\omega_{\textrm{SRF},y}$ for this metric on $X_y$.  To explain this notation: later we will define a
 \emph{semi Ricci-flat form} $\omega_{\textrm{SRF}}$ on $X$  which restricts to  $\omega_{\textrm{SRF},y}$ on $X_y$.

Our first result concerns the normalized K\"ahler-Ricci flow
\begin{equation} \label{1krfl}
\ddt{} \omega = - \Ric(\omega) - \omega, \quad \omega(0)= \omega_0,
\end{equation}
on $X$, starting at $\omega_0$.  We show that in this case we obtain exponentially fast metric collapsing of $(X, \omega(t))$ to $(B, \omega_B)$ along the K\"ahler-Ricci flow.

\begin{theorem} \label{mainthm1}
Let $\omega=\omega(t)$ be the solution of the K\"ahler-Ricci flow (\ref{1krfl}) on $X$ satisfying assumption $(*)$.  Then the following hold:
\begin{enumerate}
\item[(i)] As $t \rightarrow \infty$, $\omega$ converges exponentially fast to $\omega_B$. Namely, there exist positive constants $C, \eta>0$ such that
\begin{equation} \label{mti}
\| \omega - \omega_B\|_{C^0(X, \omega_0)} \le C e^{-\eta t}, \quad \textrm{for all } t\ge 0.
\end{equation}
\item[(ii)] The rescaled metrics $e^t \omega|_{X_y}$ converge exponentially fast to $\omega_{\emph{SRF},y}$ on each fiber $X_y$, uniformly in $y$.  Namely,
there exist positive constants $C, \eta>0$ such that
\begin{equation} \label{mtii}
\| e^t \omega|_{X_y} - \omega_{\emph{SRF},y} \|_{C^0(X_y, \omega_0|_{X_y})} \le Ce^{-\eta t}, \quad \textrm{for all } t \ge 0, \ y \in B.
\end{equation}
Moreover, if we fix $\alpha \in (0,1)$ then for each $y\in B$, $e^t\omega|_{X_y}$ converges to $\omega_{\emph{SRF},y}$ in $C^{\alpha} (X_y, \omega_0|_{X_y})$.
\item[(iii)]  As $t \rightarrow \infty$, $(X, \omega)$ converges in the Gromov-Hausdorff sense to $(B, \omega_B)$.
\end{enumerate}

\end{theorem}

Next, we consider the general case of $(X, \omega_0)$ with canonical bundle $K_X$ semiample and $0<\kappa(X) < \dim X$.  As above, we have a surjective holomorphic map $\pi : X \rightarrow B$ with connected fibers (i.e. a fiber space), but now $B$ is only an irreducible normal projective variety (possibly singular).  In addition, $\pi$ may have critical values.  If $S'\subset B$ denotes the singular set of $B$ together with the set of critical values of $\pi$, and we define $S=\pi^{-1}(S')$, then $S'$ is a proper analytic subvariety of $B$, $S$ is a proper analytic subvariety of $X$, and $\pi:X\backslash S\to B\backslash S'$ is a submersion between smooth manifolds. Therefore all fibers $X_y=\pi^{-1}(y)$ with $y\in B\backslash S'$ are smooth $n$-folds (all diffeomorphic to each other) with torsion canonical bundle.

As shown in \cite[Theorem 3.1]{ST2}, there is a smooth K\"ahler metric $\omega_B$ on $B \backslash S'$ satisfying the twisted K\"ahler-Einstein equation
$$\Ric(\omega_B) = - \omega_B + \omega_{\textrm{WP}}, \quad \textrm{on } B \backslash S',$$
for $\omega_{\textrm{WP}}$ the (smooth) Weil-Petersson form on $B \backslash S'$.  For $y \in B\backslash S'$,  let $\omega_{\textrm{SRF},y}$ be the unique Ricci-flat K\"ahler  metric on the fiber $X_y$ cohomologous to $\omega_0|_{X_y}$.

Our result is that, in this more general setting, we obtain metric collapsing of the normalized K\"ahler-Ricci flow $\omega(t)$ to $\omega_B$ on compact subsets away from the singular set $S$.

\begin{theorem} \label{mainthm2}
Let $\omega=\omega(t)$ be the solution of the K\"ahler-Ricci flow (\ref{1krfl}) on $X$.  Then the following hold:
\begin{enumerate}
\item[(i)] For each compact subset $K \subset X \backslash S$,
\begin{equation} \label{mt2i}
\| \omega - \omega_B \|_{C^0(K, \omega_0)} \rightarrow 0, \quad \textrm{as } t \rightarrow \infty.
\end{equation}
\item[(ii)] The rescaled metrics $e^t \omega|_{X_y}$ converge  to $\omega_{\emph{SRF},y}$ on each fiber $X_y$ uniformly in $y$ for compact subsets of $B \backslash S'$.  Namely, for each compact set $K' \subset B \backslash S'$
\begin{equation} \label{mt2ii}
\sup_{y \in K'} \| e^t \omega|_{X_y} - \omega_{\emph{SRF},y} \|_{C^0(X_y, \omega_0|_{X_y})} \rightarrow 0, \quad \textrm{as } t \rightarrow \infty.
\end{equation}
Moreover, if we fix $\alpha \in (0,1)$ then for each $y \in B \setminus S'$, $e^t \omega|_{X_y}$ converges to $\omega_{\emph{SRF},y}$ in $C^{\alpha}(X_y, \omega_0|_{X_y})$.
\end{enumerate}
\end{theorem}

An interesting, but still open, question is whether the convergence results of Theorems \ref{mainthm1} and \ref{mainthm2} hold with the $C^0$ norms replaced by the $C^{\infty}$ norms.  This was conjectured by Song-Tian \cite{ST} in the case of elliptic surfaces.
Moreover, one expects the convergence to be always exponentially fast. Smooth convergence has indeed been proved by Fong-Zhang \cite{FZ} in the case when the smooth fibers are tori, adapting the techniques of Gross-Tosatti-Zhang \cite{GTZ} for Calabi-Yau metrics.

\subsection{Collapsing of Ricci-flat metrics}  \label{introcy}

Let $X$ be a compact K\"ahler $(n+m)$-manifold with $c_1(X)=0$ in $H^2(X,\mathbb{R})$ (i.e. a Calabi-Yau manifold), and let $\omega_X$ be a Ricci-flat K\"ahler metric on $X$. Suppose that we have a holomorphic map $\pi:X\to Z$ with connected fibers, where $(Z,\omega_Z)$ is a compact K\"ahler manifold, with image $B=\pi(X)\subset Z$ an irreducible normal subvariety of $Z$ of dimension $m>0$. Then the induced surjective map $\pi:X\to B$ is a fiber space, and if $S'\subset B$ denotes the singular set of $B$ together with the set of critical values of $\pi$, and $S=\pi^{-1}(S')$, then $S'$ is a proper analytic subvariety of $B$, $S$ is a proper analytic subvariety of $X$, and $\pi:X\backslash S\to B\backslash S'$ is a submersion between smooth manifolds.  This is the same setup as in \cite{To, GTZ, GTZ2,HT}. The manifold $Z$ here plays only an auxiliary role, and in practice there are two natural choices for $Z$, namely $Z=B$ if $B$ is smooth, or $Z=\mathbb{CP}^N$ if $B$ is a projective variety (like in the setup of the K\"ahler-Ricci flow).

The fibers $X_y=\pi^{-1}(y)$ with $y\in B\backslash S'$ are smooth Calabi-Yau $n$-folds (all diffeomorphic to each other).  As in the previous section, we write $\omega_{\textrm{SRF},y}$ for the unique Ricci-flat K\"ahler metric on $X_y$ cohomologous to $\omega_X|_{X_y}$, for $y \in B \backslash S'$.

Write $\chi=\pi^*\omega_Z$, which is a smooth nonnegative $(1,1)$ form on $X$.  Then the cohomology class
$$\alpha_t = [\chi] + e^{-t} [\omega_X], \quad 0 \le t < \infty,$$
is K\"ahler.  By Yau's Theorem \cite{Ya} there exists a unique Ricci-flat K\"ahler metric  $\omega=\omega(t) \in \alpha_t$ on $X$.  On the other hand, the class $[\chi]$ is not K\"ahler and hence the family of metrics $\omega(t)$ must degenerate as $t\rightarrow \infty$.

Note that in \cite{To, GTZ, GTZ2,HT} the class $\alpha_t$ was replaced by $ [\chi] + t[\omega_X], 0<t\leq 1.$ This is of course only a cosmetic difference, which we have made to highlight the similarity between this setup and the one considered above for the K\"ahler-Ricci flow.

There is a K\"ahler metric $\omega_B = \chi + \ddbar v$ on $B \backslash S'$, for $v$ a unique solution of a certain complex Monge-Amp\`ere equation (see \eqref{eqnvcy}) such that
$$\Ric(\omega_B) = \omega_{\textrm{WP}}\quad \textrm{on } B \backslash S',$$
where $\omega_{\textrm{WP}}$ is the Weil-Petersson form.

We prove that  the Ricci-flat metrics $\omega \in \alpha_t$ collapse as $t \rightarrow \infty$.  Note that the limiting behavior of these metrics is similar to  solutions of the K\"ahler-Ricci flow in Theorem \ref{mainthm2} above.

\begin{theorem}  \label{mainthm3} Let $\omega=\omega(t) \in \alpha_t$ be Ricci-flat K\"ahler metrics on $X$ as described above.  Then the following hold:
\begin{enumerate}
\item[(i)] For each compact set $K \subset X \backslash S$,
\begin{equation}\label{expconv2}
\|\omega-\omega_B\|_{C^0(K,\omega_0)} \rightarrow 0, \quad \textrm{as } t\rightarrow \infty.
\end{equation}
\item[(ii)] The rescaled metrics $e^t \omega|_{X_y}$ converge to $\omega_{\emph{SRF},y}$ on each fiber $X_y$ uniformly in $y$ for compact subsets of $B \backslash S'$.  Namely, for each compact set $K' \subset B\backslash S'$,
\begin{equation}\label{expconv}
\sup_{y \in K'} \|e^{t}\omega|_{X_y}-\omega_{{\rm SRF},y}\|_{C^0(X_y,\omega_0|_{X_y})} \rightarrow 0, \quad \textrm{as } t\rightarrow \infty.
\end{equation}
Furthermore, if we fix $\alpha \in (0,1)$ then for each $y\in B\backslash S'$,
$e^{t}\omega|_{X_y}$ converges to $\omega_{{\rm SRF},y}$ in $C^\alpha(X_y,\omega_0|_{X_y})$.
\end{enumerate}
\end{theorem}

The statement that $e^{t}\omega|_{X_y}\to\omega_{{\rm SRF},y}$ (which can be improved to the smooth topology \cite{ToZ2}) solves \cite[Question 4.1]{To3} and \cite[Question 3]{To4}.
As in the case of the K\"ahler-Ricci flow, it remains an interesting open question (cf. \cite[Question 4.2]{To3}, \cite[Question 4]{To4}) whether the collapsing in \eqref{expconv2} is in the smooth topology.

Theorem \ref{mainthm3} was first proved by Gross-Wilson \cite{GW} in the case when $X$ is a $K3$ surface, $B$ is $\mathbb{CP}^1$ and the only singular fibers of $\pi$ are of type $I_1$. In fact they also proved that the collapsing in \eqref{expconv2} is in the smooth topology. Their proof is of a completely different flavor, constructing explicit almost Ricci-flat metrics on $X$ by gluing $\omega_B+e^{-t}\of$ away from $S$ with a suitable model metric near the singularities, and then perturbing this to the honest Ricci-flat metrics $\omega(t)$ (for $t$ large). In higher dimensions when there are no singular fibers ($S=\emptyset$) a similar perturbation result was obtained by Fine \cite{Fi}. However it seems hopeless to pursue this strategy in higher dimensions when there are singular fibers, since in general there is no good local model metric near the singularities. Note that such fibrations with $S=\emptyset$ are very special, thanks to the results in \cite{ToZ}.

In \cite[Theorem 7.1]{ST} Song-Tian took a different approach, and by proving a priori estimates for the Ricci-flat metrics $\omega(t)$ they showed that they collapse to $\omega_B$ in the $C^{1,\alpha}$ topology of K\"ahler potentials ($0<\alpha<1$), in the case when $X$ is $K3$ and $B$ is $\mathbb{CP}^1$. This was generalized to all dimensions by the first-named author in \cite{To}, who also proved that the fibers shrink in the $C^1$ topology of metrics.

When $X$ is projective and the smooth fibers $X_y$ are tori, Theorem \ref{mainthm3} was proved by Gross-Tosatti-Zhang in \cite{GTZ}, together with the improvement of \eqref{expconv2} to the smooth topology. The projectivity assumption was recently removed in \cite{HT}. In this case (and only in this case) the Ricci-flat metrics also have locally uniformly bounded sectional curvature on $X\backslash S$.
The assumption that the smooth fibers are tori is used crucially to derive higher-order estimates for the pullback of $\omega$ to the universal cover of $\pi^{-1}(U)$, which is biholomorphic to $U\times\mathbb{C}^n$, where $U$ is a small ball in $B\backslash S'$. It is unclear how to obtain similar estimates in our general setup.

In the setup of Theorem \ref{mainthm3} we can also derive some consequences about Gromov-Hausdorff limits of the Calabi-Yau metrics $\omega(t)$. In fact, given the uniform convergence proved in \eqref{expconv2}, we can directly apply existing arguments in the literature. For example, using the arguments in \cite[Theorem 1.2]{GTZ} we obtain the following corollary.
\begin{corollary}\label{corgh}
Assume the same setup as above, and let  $(X_{\infty},d_{\infty})$ be a metric space which is the Gromov-Hausdorff limit of $(X,\omega(t_i))$ for some sequence $t_i\to\infty$. Then there is an open dense subset  $X_0\subset X_{\infty}$ and
a homeomorphism  $\phi:B\backslash S'\to X_0$ which is a local isometry between $(B\backslash S',\omega_B)$ and $(X_0,d_{\infty})$. In particular, $(X_{\infty},d_{\infty})$ has an open dense subset which is a smooth Riemannian manifold.
\end{corollary}

Similarly, the arguments in \cite[Theorem 1.1]{GTZ2} give:
\begin{corollary}
Assume furthermore that $X$ is projective and $B$ is a smooth Riemann surface. Then $(X,\omega(t))$ have a Gromov-Hausdorff limit $(X_{\infty},d_{\infty})$ as $t\to\infty$, which is isometric to the metric
completion of $(B\backslash S',\omega_B)$ and  $X_{\infty}$ is homeomorphic to $B$. Lastly, $X_{\infty} \backslash X_0$ is a finite number of points.
\end{corollary}

This answers questions in \cite{GTZ2, To3, To4}, in this setup.\\

\subsection{Remarks on the proofs}  We end the introduction by describing briefly some  key elements in the proofs.  For simplicity of the discussion, suppose that $X$ is a product $X= E \times B$ with $\omega_E$ a Ricci-flat metric on $E$ and $\omega_B$ a K\"ahler-Einstein metric on $B$, and assume that $\omega_0$ is cohomologous to the product metric.  Write $\tilde{\omega} = \omega_B +e^{-t} \omega_E$.  It is already known \cite{FZ} that $\omega$ and $\tilde{\omega}$ are uniformly equivalent.  Our goal is to show that $\| \omega -\tilde{\omega}\|_{C^0(X,\omega_0)}$ converges to zero.  We first show that (cf. Lemma \ref{lemmaknown1}.(iv))
\begin{equation} \label{step1}
\frac{\omega^{n+m}}{\tilde{\omega}^{n+m}} \rightarrow 1, \quad \textrm{as } t\rightarrow \infty,
\end{equation}
exponentially fast.  In fact this statement is easily derived from existing arguments in the literature  \cite{ST, ST2, SW, ST4, TWY}.

The next step is the estimate (Lemma \ref{1con})
\begin{equation} \label{step2}
\tr{\omega}{\omega_B} - m \le Ce^{-\eta t},
\end{equation}
for uniform positive constants $C$ and $\eta$.  This uses a maximum principle argument similar to that in the authors' recent work \cite[Proposition 7.3]{TWY} on the Chern-Ricci flow \cite{Gi0,TW}.  Indeed, \cite{TWY} established in particular exponential convergence of solutions to the K\"ahler-Ricci flow on elliptic surfaces and, as we show here, these ideas can be applied here for the case of more general Calabi-Yau fibers.

Finally, we have a Calabi-type estimate in the fibers $X_y$ for $y \in B$  (Lemma \ref{fiberC1}) which gives
\begin{equation} \label{step3}
\| e^t\omega|_{X_y} \|_{C^1(X_y, \omega_0|_{X_y})} \le C.
\end{equation}

Then the exponential convergence of $\omega$ to $\omega_B$, and the other results of
Theorem \ref{mainthm1} are   almost a formal consequence of the three estimates (\ref{step1}), (\ref{step2}), (\ref{step3}).  When $X$ is a  holomorphic submersion rather than a product, the analogous estimates still hold.

When $\pi:X \rightarrow B$ has singular fibers then there are a number of technical complications and in particular we are not able to obtain exponential convergence (see Lemma \ref{lemmape2}.(iv)).    For the case of collapsing Ricci-flat K\"ahler metrics, there are even more difficulties.  A parabolic maximum principle argument  has to be replaced by a global argument  using the Green's function and the diameter bound of \cite{To2, ZT} (see Lemma \ref{c0lem}).  Nevertheless, the basic outlines of the proofs remain the same.

\bigskip

This paper is organized as follows. In Section \ref{sectionkrfsubmersion} we study the collapsing of the K\"ahler-Ricci flow on holomorphic submersions, and prove Theorem \ref{mainthm1}. In Section \ref{sectionkrfsing} we deal with the case with singular fibers and prove Theorem \ref{mainthm2}. The collapsing of Ricci-flat metrics on Calabi-Yau manifolds is discussed in Section \ref{sectioncy} where we give the proof of Theorem \ref{mainthm3}.\\

{\bf Acknowledgments. }We are grateful to the referee for useful comments.

\section{The K\"ahler-Ricci flow on a holomorphic submersion} \label{sectionkrfsubmersion}

In this section, we give a proof of Theorem \ref{mainthm1}.   We assume that $X$ satisfies assumption $(*)$ from the introduction.

\subsection{Holomorphic submersions} \label{subsectionsub}

We begin with some preliminary facts about submersions $\pi: X \rightarrow B$ satisfying $(*)$ as in the introduction.
The content of this section is essentially well-known (see e.g. \cite{ST, ST2, To, FZ}). We include this brief account for the sake of completeness.

Let $(X^{n+m},\omega_0)$ be a compact K\"ahler manifold which admits a submersion $\pi:X\to B$ onto a compact K\"ahler manifold $B^m$ with $c_1(B)<0$, and connected fibers
$X_y=\pi^{-1}(y)$ which are all Calabi-Yau $n$-manifolds.
Thanks to Yau's theorem \cite{Ya}, there is a smooth function $\rho_y$ on $X_y$ such that $\omega_0|_{X_y}+\ddbar \rho_y=\omega_{\textrm{SRF},y}$ is the unique
Ricci-flat K\"ahler metric on $X_y$ cohomologous to $\omega_0|_{X_y}$. If we normalize by $\int_{X_y}\rho_y\omega_0^n|_{X_y}=0$, then $\rho_y$ varies smoothly in $y$ and defines a smooth function $\rho$ on $X$ and we let
$$\of=\omega_0+\ddbar\rho.$$
This is called a semi Ricci-flat form, because it restricts to a Ricci-flat K\"ahler metric on each fiber $X_y$.  It was first introduced by Greene-Shapere-Vafa-Yau \cite{GSVY} in the context of elliptically fibered $K3$ surfaces (in which case it is a semiflat form).
Note that while $\of$ need not be positive definite, nevertheless for any K\"ahler metric $\chi$ on $B$ we have that $\of^n\wedge\pi^*\chi^m$ is a strictly positive volume form on $X$.
Now recall that the relative pluricanonical bundle of $\pi$ is $K_{X/B}^\ell=K_X^\ell\otimes (\pi^*K_B^\ell)^*$, where $\ell$ is any positive integer. Thanks to the projection formula, we have
$$\pi_* (K_X^\ell)=(\pi_* (K_{X/B}^\ell))\otimes K_B^\ell,$$
and when restricted to any fiber $X_y$, $K^\ell_{X/B}|_{X_y}\cong K^\ell_{X_y}$. The fact that $X_y$ is Calabi-Yau implies that there is a positive integer $\ell$ such that $K_{X_y}^\ell$ is trivial, for all $y\in B$, and from now on we fix such a value of $\ell$.
Then Grauert's theorem on direct images \cite[Theorem I.8.5]{bhpv} shows that
$$L:=\pi_* (K_{X/B}^\ell),$$
is a line bundle on $B$ . Since all the fibers of $\pi$ have trivial $K_{X_y}^\ell$, it follows
that $K_X^\ell\cong \pi^*\pi_*(K_X^\ell)$ (see \cite[Theorem V.12.1]{bhpv}).
We conclude that
\begin{equation}\label{canon}
K_X^\ell=\pi^*(K_B^\ell\otimes L).
\end{equation}
The line bundle $L= \pi_* (K_{X/B}^\ell)$ is Hermitian semipositive, and the Weil-Petersson form $\omega_{\mathrm{WP}}$ on $B$
is a semipositive representative of $\frac{1}{\ell}c_1(L)$. This is defined as follows.
Choose any local nonvanishing holomorphic section $\Psi_y$ of $\pi_* (K_{X/B}^\ell)$, i.e. a
family $\Psi_y$ of nonvanishing holomorphic $\ell$-pluricanonical forms on $X_y$, which vary holomorphically in $y$.
The forms $\Psi_y$ are defined on the whole fiber $X_y$, but only for $y$ in a small ball in $B$.
Then the Weil-Petersson $(1,1)$-form on this ball
$$\omega_{\mathrm{WP}}=-\mn\de_y\db_y\log \left((\mn)^{n^2}\int_{X_y} (\Psi_y\wedge\ov{\Psi}_y)^{\frac{1}{\ell}}\right),$$
is actually globally defined on $B$ (this is because the fibers have trivial $\ell$-pluricanonical bundle,
so different choices of $\Psi_y$ differ by multiplication by a local holomorphic function in $y$,
and so $\omega_{\mathrm{WP}}$ is globally defined) and is semipositive definite (this was originally proved
by Griffiths \cite{Gr}, and is a by now standard calculation in Hodge theory). One can view $\omega_{\mathrm{WP}}$ as the pullback of
the Weil-Petersson metric from the moduli space of polarized Calabi-Yau fibers of $\pi$, thanks to \cite{FS, Ti, Tod}.
From this construction, we see that
$\ell\omega_{\mathrm{WP}}$ is the curvature form of a singular $L^2$ metric on $\pi_* (K_{X/B}^\ell)$, so $[\omega_{\mathrm{WP}}]=\frac{1}{\ell}c_1(L)$.  Notice that here and henceforth, we omit the usual factor of $2\pi$ in the definition of the first Chern class.

From \eqref{canon} we see that $K_X^\ell$ is the pullback of an ample line bundle, hence $K_X$ is semiample.
It follows also that
\begin{equation}\label{c1}
c_1(X)=\pi^*c_1(B)-[\pi^*\omega_{\mathrm{WP}}].
\end{equation}
In our setup the class $-c_1(B)$ is K\"ahler, and therefore so is the class
$-c_1(B)+[\omega_{\mathrm{WP}}]$. Fix a K\"ahler metric $\chi$ in this class. From \eqref{c1} we see that
$-\pi^*\chi$ is cohomologous to $c_1(X)$, and so there exists a unique smooth positive volume form $\Omega_1$ on $X$ with
$$\ddbar \log (\Omega_1)=\chi, \quad \int_X \Omega_1=\binom{n+m}{n}\int_X \omega_0^n\wedge \chi^m,$$ where from now on we abbreviate $\pi^*\chi$ by $\chi$.

We now show the existence of a unique twisted K\"ahler-Einstein metric on $B$.
\begin{theorem}\label{genke}
There is a unique K\"ahler metric $\omega_B$ on $B$ which satisfies
\begin{equation}\label{gke}
\Ric(\omega_B)=-\omega_B+\omega_{\mathrm{WP}}.
\end{equation}
\end{theorem}
\begin{proof}
The proof is identical to \cite[Theorem 3.1]{ST} or \cite[Theorem 3.1]{ST2}, \cite[Section 4]{To}, so we will just briefly indicate the main points.
First, set
$$F=\frac{\Omega_1}{\binom{n+m}{n}\omega_{\textrm{SRF}}^n\wedge\chi^m},$$
which is a positive function on $X$. The argument on \cite[p.445]{To} shows that for every $y\in B$ the restriction $F|_{X_y}$ is constant, and so $F$ is the pullback of a function on $B$.
Now solve the complex Monge-Amp\`ere equation on $B$
\begin{equation}\label{gma}
(\chi+\ddbar v)^m=Fe^v \chi^m,
\end{equation}
where $\omega_B=\chi+\ddbar v$ is a K\"ahler metric on $B$. This equation can be solved thanks to work of Aubin \cite{Au} and Yau \cite{Ya}.
We claim that $\omega_B$ satisfies \eqref{gke} (the converse implication, that every solution of \eqref{gke} must also satisfy \eqref{gma}, and therefore must equal $\omega_B$, is a simple exercise). Indeed,
$$\Ric(\omega_B)=\Ric(\chi)-\ddbar v-\ddbar \log F,$$
so it is enough to prove that
\begin{equation}\label{curv}
-\ddbar \log F = -\chi-\Ric(\chi)+\omega_{\mathrm{WP}}.
\end{equation}
This can be proved by the same calculation as in \cite[Proposition 4.1]{To}.
\end{proof}

Now that we have our twisted K\"ahler-Einstein metric $\omega_B$ on $B$, we can use \eqref{c1} again to see that
$-\pi^*\omega_B$ is cohomologous to $c_1(X)$. Define now a smooth positive volume form
\begin{equation}\label{crux3}
\Omega=\binom{n+m}{n}\omega_{\textrm{SRF}}^n\wedge\omega_B^m.
\end{equation}
on $X$. A simple calculation as in \cite[Proposition 4.1]{To} shows that
$$\ddbar \log (\Omega)= \omega_B,$$
where from now on we abbreviate $\pi^*\omega_B$ by $\omega_B$.

\subsection{Preliminary estimates for the K\"ahler-Ricci flow}  Let now $\omega=\omega(t)$ be a solution of the K\"ahler-Ricci flow
\begin{equation} \label{krf11}
\frac{\de}{\de t}\omega=-\Ric(\omega)-\omega,\quad \omega(0)=\omega_0.
\end{equation}
We know that a solution exists for all time \cite{TZ}, because $K_X$ is semiample (see Section \ref{subsectionsub}).
Define reference forms $\ti{\omega}=\ti{\omega}(t)$ by
$$\ti{\omega}=e^{-t}\of+(1-e^{-t})\omega_B,$$
for $\of$ and $\omega_B$ as above.

There exists a uniform constant $T_I \ge 0$ such that
for $t\geq T_I$ we have that $\ti{\omega}$ is K\"ahler.
Then the K\"ahler-Ricci flow is equivalent to the parabolic complex Monge-Amp\`ere equation
\begin{equation} \label{pcma1}
\frac{\de}{\de t}\vp=\log\frac{e^{nt}(\ti{\omega}+\ddbar\vp)^{n+m}}{\Omega}-\vp, \quad \varphi(0) = -\rho, \quad \ti{\omega}+\ddbar\vp>0,
\end{equation}
where $\rho$ and $\Omega$ are defined as above.   Indeed, if $\varphi$ solves (\ref{pcma1}) then $\omega = \ti{\omega} + \ddbar \vp$ solves (\ref{krf11}).  Conversely, if $\omega$ solves (\ref{krf11}) then we can write $\omega= \ti{\omega} + \ddbar \varphi$ with $\varphi$ solving (\ref{pcma1}).

We have the following preliminary estimates for the K\"ahler-Ricci flow, most of which are already known by the results in \cite{ST, SW, FZ, Gi}.

\begin{lemma} \label{lemmaknown1} Let $\omega=\omega(t)$ solve the K\"ahler-Ricci flow as above, and write $\varphi=\varphi(t)$ for the solution of (\ref{pcma1}).  Then there exists a uniform constant $C>0$ such that  the following hold.
\begin{enumerate}
\item[(i)] $\displaystyle{C^{-1} \ti{\omega} \le \omega \le C \ti{\omega}}$, on $X \times [T_I, \infty)$, and $\omega_B \le C \omega$ on $X \times [0,\infty)$.
\item[(ii)] $\displaystyle{|\varphi | \le C(1+ t)e^{-t}}$ on $X \times [0,\infty)$.
\item[(iii)] The scalar curvature $R$ of $\omega$ satisfies $\displaystyle{|R| \le C}$ on $X \times [0,\infty)$.
\item[(iv)] For any $0<\ve < 1/2$ there exists $C_{\ve}$ such that  $\displaystyle{ | \dot{\varphi} |\le C_{\ve} e^{-\ve t}}$ and $\displaystyle{ | \varphi+ \dot{\varphi} |\le C_{\ve} e^{-\ve t}}$ on $X \times [0,\infty)$.
\end{enumerate}
\end{lemma}
\begin{proof}
For (i), the estimate $\omega_B \le C\omega$ follows from the Schwarz Lemma argument of Song-Tian \cite{ST, Ya2}.  The bound $C^{-1} \ti{\omega} \le \omega \le C \ti{\omega}$
 was proved by Fong-Zhang \cite[Theorem 1.1]{FZ} (see also \cite{ST} for the case of elliptic surfaces).

Part (ii) follows from exactly the same argument as in \cite[Lemma 4.1]{Gi}, which is a generalization of \cite[Lemma 6.7]{SW} (see also \cite[Lemma 3.4]{TWY}, \cite[Theorem 4.1]{FZ}).
Indeed, we claim that for $t\geq T_I$,
\begin{equation} \label{claimOmega}
\left| e^t  \log \frac{e^{nt} \tilde{\omega}^{n+m}}{\Omega}\right|\le C.
\end{equation}
To see the claim, we have, using \eqref{crux3},
\begin{align*}
\lefteqn{e^t  \log \frac{e^{nt} \tilde{\omega}^{n+m}}{\Omega} } \\ = {} & e^t \log \frac{e^{nt}(\binom{n+m}{n}e^{-nt}(1-e^{-t})^m\of^n\wedge\omega_B^m +\dots+e^{-(n+m)t}\of^{n+m})}{\binom{n+m}{n}\omega_{\textrm{SRF}}^n\wedge\omega_B^m} \\
= {} & e^t \log (1+ O(e^{-t})),
\end{align*}
which is bounded.

Define now $Q=e^t \varphi +At$, for $A$ a large positive constant to be determined. Then
\begin{equation}
\frac{\partial Q}{\partial t} = e^t \log \left( \frac{e^{nt} (\tilde{\omega}+ \ddbar \varphi)^{n+m}}{\Omega} \right)+ A.
\end{equation}
We wish to bound $Q$ from below.  Suppose that $(x_0, t_0)$ is a point with $t_0 > T_I$ at which $Q$ achieves a minimum.  At this point we have
\begin{align*}
0 \ge \frac{\partial Q}{\partial t} \ge {} &
e^t\log\frac{e^{nt}\ti{\omega}^{n+m}}{\Omega}+A \ge   -C'+A,
\end{align*}
for a uniform $C'$, thanks to \eqref{claimOmega}.
 Choosing $A> C'$ gives contradiction.
Hence $Q$ is bounded from below and it follows that $\varphi\geq -C(1+t)e^{-t}$ for a uniform
$C$. The upper bound for $\varphi$ is similar.

Part (iii) is due to Song-Tian \cite{ST3}, who proved a uniform scalar curvature bound along the K\"ahler-Ricci flow whenever the canonical bundle of $X$ is semiample.

Part (iv) follows from the argument of \cite[Lemma 6.4]{TWY} (see also \cite[Lemma 6.8]{SW}).  We include the argument for the sake of completeness.  First note that $\dot{\varphi}$ satisfies
$$\ddt{} \dot{\varphi} = - R-m -\dot{\varphi}.$$
Then the bound on scalar curvature (iii) implies  that $\dot{\varphi}$ is bounded (in fact the bound on $\dot{\varphi}$ is proved first in \cite{ST3} in order to bound scalar curvature).  Hence we have a bound $|\ddt \dot{\vp}| \le C_0$.  Now fix $\ve \in (0,1/2)$ and suppose for a contradiction that the estimate $\dot{\varphi} \le C_{\ve} e^{-\ve t}$ fails.  Then there exists a sequence $(x_k, t_k) \in X \times [0,\infty)$ with $t_k \rightarrow \infty$ and $\dot{\varphi}(x_k, t_k) \ge k e^{-\ve t_k}$.  Define $\gamma_k = \frac{k}{2C_0}e^{-\ve t_k}$.  Working at the point $x_k$, it follows that $\dot{\varphi} \ge \frac{k}{2} e^{-\ve t_k}$ on  $[t_k, t_k+\gamma_k]$.  Then using the bound (ii),
$$C(1+t_k)e^{-t_k} \ge \int_{t_k}^{t_k+\gamma_k} \dot{\vp} \, dt \ge  \gamma_k \frac{k}{2} e^{-\ve t_k} = \frac{k^2}{4C_0} e^{-2\ve t_k}.$$
Since $2\ve<1$ we obtain a contradiction as $k \rightarrow \infty$.  The lower bound $\dot{\varphi} \ge -C_{\ve} e^{-\ve t}$ is similar.  The bound on $|\varphi+\dot{\varphi}|$ follows from this and part (ii).
\end{proof}

The following lemma is easily derived from Song-Tian \cite{ST}.

\begin{lemma} \label{lemmaknown2} We have
\begin{equation} \label{epdp}
\left(\frac{\de}{\de t}-\Delta\right)(\vp+\dot{\vp})=\tr{\omega}{\omega_B}-m,
\end{equation}
and
\begin{equation} \label{schwarz1}
\left(\frac{\de}{\de t}-\Delta\right)\tr{\omega}{\omega_B}\leq \tr{\omega}{\omega_B}+C(\tr{\omega}{\omega_B})^2 \le C.
\end{equation}
\end{lemma}
\begin{proof}
The evolution equation (\ref{epdp}) follows from the two well-known calculations:
$$\left(\frac{\de}{\de t}-\Delta\right)\vp=\dot{\vp}-(n+m)+\tr{\omega}{\ti{\omega}}$$
and
$$\left(\frac{\de}{\de t}-\Delta\right)\dot{\vp}=\tr{\omega}{(\omega_B-\ti{\omega})}+n-\dot{\vp}.$$
The first inequality  of (\ref{schwarz1}) is the Schwarz Lemma computation  from \cite[Theorem 4.3]{ST}.  The second inequality of (\ref{schwarz1}) follows from part (i) of Lemma \ref{lemmaknown1}.
\end{proof}

\subsection{Two elementary lemmas}

In this section, we state and prove two elementary lemmas which we will need later in our proofs of the collapsing estimates.  The first is as follows:

\begin{lemma}\label{1conve}
Consider a smooth function $f: X \times [0,\infty) \rightarrow \mathbb{R}$ which satisfies the following conditions:
\begin{itemize}
\item[(a)] There is a constant $A$ such that for all $y\in B$ and all $t\geq 0$
$$|\nabla (f|_{X_y})|_{\omega_0|_{X_y}}\leq A.$$
\item[(b)] For all $y\in B$ and all $t\geq 0$ we have
$$\int_{x \in X_y}f(x,t) \omega_{{\rm SRF},y}^n=0.$$
\item[(c)] There exists a function $h: [0, \infty) \rightarrow [0,\infty)$ such that $h(t) \rightarrow 0$ as $t \rightarrow \infty$ such that
$$\sup_{x\in X} f(x,t)\leq h(t),$$
for all $t\geq 0$.
\end{itemize}
Then  there is a constant $C$ such that $$\sup_{x\in X} |f(x,t)|\leq C (h(t))^{1/(2n+1)}$$ for all $t$ sufficiently large.
\end{lemma}
\begin{proof}
Arguing by contradiction, suppose that there exist $t_k \to\infty$ and  $x_k \in X$ with
$$\sup_X |f(t_k)|=|f(x_k, t_k)|\geq k (h(t_k))^{1/(2n+1)},$$
for all $k$. We may also assume that $t_k$ is sufficiently large so that the quantity $k (h(t_k))^{1/(2n+1)}$ is less than or equal to $1$ for all $k$.
Let $y_k=\pi(x_k)$, so $x_k \in X_{y_k}$. Thanks to assumption (c), for all $k$ large, we must have
\begin{equation}\label{1e2}
f( x_k, t_k)\leq - k (h(t_k))^{1/(2n+1)}.
\end{equation}
This together with assumption (a) implies that for all $x$ in the $\omega_0$-geodesic ball $B_r(x_k)$ in $X_{y_k}$ centered at $x_k$ and of radius
$$r=\frac{ k (h(t_k))^{1/(2n+1)}}{2A}\leq \frac{1}{2A},$$ we have
$$f(x,t_k)\leq -\frac{ k (h(t_k))^{1/(2n+1)}}{2}.$$
Therefore, using assumptions (b) and (c),
$$0=\int_{x\in X_{y_k}}f(x,t_k)\omega_{{\rm SRF},y_k}^n\leq -\frac{ k (h(t_k))^{1/(2n+1)}}{2} \int_{B_r(x_k)} \omega_{{\rm SRF},y_k}^n+Ch(t_k).$$
Now the metrics $\omega_{{\rm SRF},y_k}$ are all uniformly equivalent to each other, and in particular,
$$\int_{B_r(x_k)} \omega_{{\rm SRF},y_k}^n\geq C^{-1} k^{2n} (h(t_k))^{2n/(2n+1)},$$
using that $k (h(t_k))^{1/(2n+1)} \leq 1$ and that small balls for $\omega_{{\rm SRF},y_k}$ have volume comparable to Euclidean balls.
Thus we get
$$k^{2n+1}\leq C,$$
which gives a contradiction when $k$ is large.
\end{proof}

\begin{remark} \label{remarkcpt}
Note that Lemma \ref{1conve} also holds with the same proof if $X$ is replaced by $K=\pi^{-1}(K')$ for $K'$ a compact subset of $B$.  Hence the result holds for our more general holomorphic  fiber spaces $\pi: X \rightarrow B$ as discussed in the introduction, if we restrict to  $\pi^{-1}(K')$ for a compact set $K' \subset B \backslash S'$.
\end{remark}

The second lemma is an elementary observation from linear algebra.

\begin{lemma} \label{lemmamatrix}
Let $A$ be an $N \times N$ positive definite symmetric matrix.  Assume that there exists $\ve \in (0,1)$ with
$$  \emph{tr} \, A   \le N + \ve, \quad   \det A   \ge 1-\ve.$$
Then there exists a constant $C_N$ depending only on $N$ such that
$$\| A- I \| \le C_N \sqrt{\ve},$$
where $\| \cdot \|$ is the Hilbert-Schmidt norm, and $I$ is the $N\times N$ identity matrix.
\end{lemma}
\begin{proof}  We may assume that $N \ge 2$.
Let $\lambda_1, \ldots, \lambda_N$ be the eigenvalues of $A$.  Define the (normalized) elementary symmetric polynomials $S_k$ by
$$S_k = {\binom{N}{k}}^{-1} \sum_{1 \le i_1 < \cdots < i_k \le N} \lambda_{i_1} \cdots \lambda_{i_k}, \quad \textrm{for } k=1, \ldots, N.$$
By assumption, $S_1 \le 1+ N^{-1} \ve$ and $S_N \ge 1-\ve$.
The Maclaurin inequalities
$$S_1 \ge \sqrt{S_2} \ge S_N^{1/N},$$
imply that $|S_1-1|+|S_2-1| \le C \ve$ for $C$ depending only on $N$.
Compute
$$\| A-I \|^2 = \sum_{j=1}^N (\lambda_j-1)^2 = N^2 S_1^2 -2NS_1- N(N-1) S_2 +N \le C' \ve,$$
for $C'$ depending only on $N$, and the result follows.
\end{proof}

\subsection{A local Calabi  estimate for collapsing metrics}  Later we will need a local Calabi estimate for solutions of the K\"ahler-Ricci flow which are collapsing exponentially in specified directions.  The Calabi estimate for the elliptic complex Monge-Amp\`ere equation, giving third order estimates for the potential function depending on the second order estimates, first appeared in the paper of Yau \cite{Ya}, and was used in the parabolic case by Cao \cite{Ca}.
In the following we use a formulation due to Phong-Sesum-Sturm \cite{PSS}, together with a cut-off function  argument similar to that used in \cite{ShW}.
\begin{proposition}\label{1localest}
Let $B_1(0)$ be the unit polydisc in $\mathbb{C}^{n+m}$ and let $\omega_E^{(n+m)} = \sum_{k=1}^{m+n} \sqrt{-1} dz^k \wedge d\ov{z}^k$ be the Euclidean metric.
Write $\mathbb{C}^{n+m}=\mathbb{C}^{m}\oplus\mathbb{C}^{n}$ and define $\omega_{E,t}=\omega_{E}^{(m)} + e^{-t}\omega^{(n)}_E$,  for $\omega_{E}^{(m)}$ and $\omega^{(n)}_E$ Euclidean metrics on the two factors $\mathbb{C}^m$ and $\mathbb{C}^n$ respectively.
Assume that $\omega=\omega(t)$ is a solution of the K\"ahler-Ricci flow
$$\frac{\de \omega}{\de t}=-\Ric(\omega)-\omega,  \quad \omega(0)=\omega_0,$$
on $B_1(0)\times [0,\infty)$ which satisfies
\begin{equation}\label{1assum}
A^{-1}\omega_{E,t}\leq \omega \leq A\, \omega_{E,t},
\end{equation}
for some positive constant $A$.
Then there is a constant $C$ that depends only on $n,m,A$  and $\omega_0$ such that for all $t\geq 0$ on $B_{1/2}(0)$ we have
\begin{equation}\label{1bo}
S=|\nabla^E g|^2_{g} \leq Ce^t,
\end{equation}
where $g$ is the metric associated to $\omega$ and $\nabla^E$ is the covariant derivative of $\omega_E^{(n+m)}$.
\end{proposition}

\begin{proof}
In the following $C$ will denote a generic constant that depends
only on $n,m, A$  and $\omega_0$. For simplicity of notation, we write simply $\omega_E=\omega_E^{(n+m)}$.
Following Yau \cite{Ya}, define
$$S=|\nabla^E g|_g^2=g^{i\ov{\ell}} g^{p\ov{k}} g^{j\ov{q}} \nabla^E_i g_{j\ov{k}} \nabla^E_{\ov{\ell}}g_{p\ov{q}}.$$
If $T_{ij}^k$ denotes the difference of the Christoffel symbols of $\omega$ and $\omega_E$, then $T$ is a tensor and it is easy to see that $S=|T|^2_g$. Since $\omega_{E,t}$ is also a flat metric, the quantity $S$ is the same if we use its connection $\nabla^{E,t}$.
The parabolic Calabi computation gives (cfr. \cite{PSS})
$$\left(\frac{\de}{\de t}-\Delta\right) S \leq S-|\nabla T|^2 -|\ov{\nabla} T|^2,$$
where here and in the rest of the proof we suppress the subscript $g$ from the norms.

Take $\psi$ a nonnegative smooth cutoff function supported
in $B_1(0)$ with $\psi\equiv 1$ on $B_{1/2}(0)$ and with $$\sqrt{-1} \partial \psi \wedge \ov{\partial} \psi \le C \omega_E, \qquad -C \omega_E \le \ddbar (\psi^2) \le C \omega_E.$$
From \eqref{1assum} it then follows that
$$|\nabla \psi|^2 \leq Ce^t, \quad \Delta (\psi^2)\geq -Ce^t.$$
We can then compute
\[\begin{split}
\left(\frac{\de}{\de t}-\Delta \right)(\psi^2 S)&\leq \psi^2\left(\frac{\de}{\de t}-\Delta \right)S+C Se^t+2|\langle \nabla \psi^2,\nabla S\rangle|\\
&\leq \psi^2S-\psi^2(|\nabla T|^2+|\ov{\nabla} T|^2)+C Se^t+2|\langle \nabla \psi^2,\nabla S\rangle|.
\end{split}\]
On the other hand, using the Young inequality
\[\begin{split}
2|\langle \nabla \psi^2,\nabla S\rangle|
&=4\psi|\langle \nabla\psi, \nabla |T|^2\rangle|\leq 4\psi |\nabla \psi|\cdot | \nabla |T|^2|\\
&\leq \psi^2(|\nabla T|^2+|\ov{\nabla} T|^2)+CS|\nabla \psi|^2\\
&\leq \psi^2(|\nabla T|^2+|\ov{\nabla} T|^2)+CSe^t,
\end{split}\]
and so
$$\left(\frac{\de}{\de t}-\Delta\right) (\psi^2 S)\leq CSe^t,$$
$$\left(\frac{\de}{\de t}-\Delta\right) (e^{-t}\psi^2 S)\leq CS.$$
On the other hand, the second order estimate of Cao \cite{Ca} (the parabolic version of the estimate of \cite{Ya, Au}) gives
\[
\begin{split}
\left(\frac{\de}{\de t}-\Delta\right) \tr{\omega_{E,t}}{\omega}= {} & -\tr{\omega_{E,t}}{\omega}+e^{-t}g_{E,t}^{i\ov{q}} \, g_{E,t}^{p\ov{j}}\, (g^{(n)}_E)_{p\ov{q}} \, g_{i\ov{j}} \\
& {} -g_{E,t}^{i\ov{\ell}}\, g^{p\ov{j}} g^{k\ov{q}} \nabla^E_i g_{k\ov{j}} \nabla^E_{\ov{\ell}}g_{p\ov{q}}\\
 \leq {} & -A^{-1}S,
\end{split}
\]
using \eqref{1assum}. It follows that if we take $C_0$ large enough, then we have
$$ \left(\frac{\de}{\de t}-\Delta \right)\left(e^{-t}\psi^2 S+C_0\tr{\omega_{E,t}}{\omega}\right)  < 0.$$
It follows that the maximum of $e^{-t}\psi^2 S+C_0\tr{\omega_{E,t}}{\omega}$ on $\ov{B_1(0)}\times [0,T]$ (for any finite $T$) can only occur at $t=0$ or on the boundary of
$\ov{B_1(0)}$, where $\psi=0$. Since by \eqref{1assum} we have $\tr{\omega_{E,t}}{\omega}\leq  (n+m)A$, we conclude that
$\sup_{B_{1/2}(0)}S\leq Ce^t,$ as required.
\end{proof}

\subsection{Collapsing estimates for the K\"ahler-Ricci flow} In this section, we give the proof of Theorem \ref{mainthm1}, which establishes the metric collapsing for the K\"ahler-Ricci flow in the case of a holomorphic submersion.  We will make use of the lemmas and observations  above.

 We begin with a simple but crucial lemma, which is a generalization of \cite[Proposition 7.3]{TWY}.  We will need it later.

\begin{lemma}\label{1con}
For any $\eta \in (0,1/4)$, there is a constant $C$ such that
$$\tr{\omega}{\omega_B}-m\leq Ce^{-\eta t}.$$
\end{lemma}
\begin{proof}
Let
$$Q=e^{\eta t}(\tr{\omega}{\omega_B}-m)-e^{2\eta t}(\vp+\dot{\vp}).$$
Then from Lemma \ref{lemmaknown2},
$$\left(\frac{\de}{\de t}-\Delta\right)Q\leq \eta e^{\eta t}(\tr{\omega}{\omega_B}-m)+ Ce^{\eta t} - 2\eta e^{2\eta t}(\vp+\dot{\vp})-
e^{2\eta t}(\tr{\omega}{\omega_B}-m).$$
From Lemma \ref{lemmaknown1},  we have $\tr{\omega}{\omega_B}\leq C$ and $|e^{2\eta t}(\vp+\dot{\vp})|\leq C.$  Then
$$\left(\frac{\de}{\de t}-\Delta\right)Q\leq Ce^{\eta t}-
e^{2\eta t}(\tr{\omega}{\omega_B}-m),$$
and so at a maximum point of $Q$ we have an upper bound for $e^{\eta t}(\tr{\omega}{\omega_B} -m)$ and hence for $Q$.   The result follows.
\end{proof}

We will now establish part (ii) of Theorem \ref{mainthm1}.  We have  the following estimate in the fibers $X_y$ of $\pi$.

\begin{lemma} \label{fiberC1}
There is a constant $C>0$ such that for all $t \ge 0$ and all $y \in B$,
$$\| e^t \omega|_{X_y} \|_{C^1(X_y, \omega_0|_{X_y})} \le C, \quad e^t \omega|_{X_y} \ge C^{-1} \omega_0|_{X_y}.$$
\end{lemma}
\begin{proof}
Fix a point $y\in B$ and a point $x\in X_{y}$.  Since $\pi:X \rightarrow B$ is a holomorphic submersion, we can find a  $U\subset X$ a local holomorphic product coordinate chart for $\pi$ centered at $x$,
which equals the unit polydisc in $\mathbb{C}^{n+m}$ (see e.g. \cite[p.60]{Ko}).
Thanks to part (i) of Lemma \ref{lemmaknown1}, we see that the hypotheses of Proposition \ref{1localest} are satisfied, and so we conclude that on the half-sized polydisc we have
$$| \nabla^E \omega|^2_{\omega} \leq Ce^t,$$
where $\nabla^E$ is the covariant derivative of the Euclidean metric $\omega_E$ on $U$.  We may assume that $\omega_E$ is uniformly equivalent to $\omega_0$ on $U$. Again using Lemma \ref{lemmaknown1}.(i), we see that on $U$ the metric $\omega|_{X_y}$ is uniformly equivalent to
$e^{-t}\omega_E|_{X_y}$, and so on $U$ we have (cf. \cite[Section 3]{To} or \cite[Lemma  3.6.9]{SW})
$$| \nabla^E (e^t\omega|_{X_y})|^2_{\omega_E}=e^{-t}|\nabla^E (\omega|_{X_y})|^2_{e^{-t}\omega_E}\qquad \qquad$$ $$ \leq Ce^{-t}|\nabla^E(\omega|_{X_y})|^2_\omega\leq Ce^{-t}|\nabla^E \omega|^2_{\omega}\leq C.$$
Furthermore, all these estimates are uniform in $y\in B$.
The uniform $C^1$ bound as well as the uniform lower bound for $e^t \omega|_{X_{y}}$
follows at once.
\end{proof}

Thanks to Lemma \ref{fiberC1}, for any fixed $y\in B$ and $0<\alpha<1$, given any sequence $t_k\to\infty$ we can extract a subsequence such that
$$e^{t_k}\omega|_{X_y}\to\omega_{\infty,y},$$
in $C^\alpha$ on $X_y$, where $\omega_{\infty,y}$ is a $C^\alpha$ K\"ahler metric on $X_y$ cohomologous to $\omega_0|_{X_y}$.

We can now prove the last statement in Theorem \ref{mainthm1}.(ii).
\begin{lemma}\label{1conv}
Given any $y\in B$, as $t\to\infty $ we have that
$$e^{t}\omega|_{X_y}\to\omega_{{\rm SRF},y},$$
in $C^\alpha$ on $X_y$, for any $0<\alpha<1$.
\end{lemma}

\begin{proof}
Write
\[\begin{split}
(e^t\omega|_{X_y})^n&= e^{nt} \frac{(\omega|_{X_y})^n}{(\omega_{{\rm SRF},y})^n}(\omega_{{\rm SRF},y})^n=e^{nt}\frac{\omega^n\wedge\omega_B^{m}}{\omega_{{\rm SRF}}^n\wedge\omega_B^m}(\omega_{{\rm SRF},y})^n\\
&=e^{nt}\binom{n+m}{n}\frac{\omega^n\wedge\omega_B^{m}}{\Omega}(\omega_{{\rm SRF},y})^n\\
&=e^{\vp+\dot{\vp}}\binom{n+m}{n}\frac{\omega^n\wedge\omega_B^{m}}{\omega^{n+m}}(\omega_{{\rm SRF},y})^n.
\end{split}\]
Consider the function on $X$, depending on $t$, defined by
$$f=e^{\vp+\dot{\vp}}\binom{n+m}{n}\frac{\omega^n\wedge\omega_B^{m}}{\omega^{n+m}},$$
which when restricted to the fiber $X_y$ equals
\begin{equation}\label{1fib}
f|_{X_y}=\frac{(e^t\omega|_{X_y})^n}{(\omega_{{\rm SRF},y})^n}.
\end{equation}
We have that
\begin{equation}\label{1intt}
\int_{X_y}f(\omega_{{\rm SRF},y})^n=\int_{X_y} (e^t\omega|_{X_y})^n=\int_{X_y}(\omega_0|_{X_y})^n=\int_{X_y}(\omega_{{\rm SRF},y})^n,
\end{equation}
so that $f-1$ satisfies condition (b) of Lemma \ref{1conve}. From Lemma \ref{lemmaknown1}.(iv), for $t$ large we have that
\begin{equation}\label{1expd}
|e^{\vp+\dot{\vp}}-1|\leq Ce^{-\eta t},
\end{equation}
 as long as $\eta<1/2$.
On the other hand, if at a point on $X_y$ we choose coordinates so that $\omega$ is the identity and $\omega_B$ is diagonal with eigenvalues $\lambda_1,\dots,\lambda_m$, then
at that point
\begin{equation}\label{1amgm}
\binom{n+m}{n}\frac{\omega^n\wedge\omega_B^{m}}{\omega^{n+m}}=\prod_j \lambda_j\leq (\sum_j\lambda_j/m)^{m}=(\tr{\omega}{\omega_B}/m)^{m}.
\end{equation}
By Lemma \ref{1con}, $\sup_X\tr{\omega}{\omega_B}/m\leq 1+Ce^{-\eta t},$ for $\eta<1/4$.  Then \eqref{1expd} and \eqref{1amgm} imply
that $\sup_X f\leq 1+Ce^{-\eta t}$, for some $\eta>0$. This shows that $f-1$ satisfies condition (c) of Lemma \ref{1conve}.
Lastly, $f-1$ also satisfies condition (a) because of \eqref{1fib} and the estimates in Lemma  \ref{fiberC1}. Therefore Lemma \ref{1conve} shows that $f$ converges to $1$ uniformly on $X$ and exponentially fast as $t\to\infty$, i.e.
\begin{equation}\label{1volcon}
\|(e^{t}\omega|_{X_y})^n-\omega_{{\rm SRF},y}^n\|_{C^0(X_y,\omega_0|_{X_y})}\leq Ce^{-\eta t},
\end{equation}
for some $\eta>0$, for all $y\in B$.

If we choose a sequence $t_k\to\infty$ such that $e^{t_k}\omega|_{X_y}\to\omega_{\infty,y},$ in $C^\alpha$ on $X_y$, then the $C^\alpha$ metric $\omega_{\infty,y}$ on $X_y$ satisfies
$$\omega_{\infty,y}^n=\omega_{{\rm SRF},y}^n.$$
Standard bootstrap estimates give that $\omega_{\infty,y}$ is a smooth K\"ahler metric, and the uniqueness of the solution of this equation (originally due to Calabi) implies that
$\omega_{\infty,y}=\omega_{{\rm SRF},y}$. From this it  follows that $e^{t}\omega|_{X_y}\to\omega_{{\rm SRF},y}$ in $C^\alpha$, as required.
\end{proof}

Note that in the previous lemma, since we only have derivative bounds on $\omega$ in the fiber directions, we are not able to conclude that the $C^{\alpha}$ convergence is uniform in $y \in B$.  However, we can now establish uniformity of the fiberwise convergence in the $C^0$ norm and complete the
 proof of Theorem \ref{mainthm1}.(ii).  Namely, we will prove the estimate \eqref{mtii}: $$\| e^t \omega|_{X_y} - \omega_{\textrm{SRF},y} \|_{C^0(X_y, \omega_0|_{X_y})} \le Ce^{-\eta t}, \quad \textrm{for all } y \in B.$$

\begin{proof}[Proof of (\ref{mtii})]
Consider the function
$$f|_{X_y}=\frac{(e^t\omega|_{X_y})\wedge \omega_{{\rm SRF},y}^{n-1}}{\omega_{{\rm SRF},y}^n}.$$
While this is defined on $X_y$, it is clearly smooth in $y$ and so defines a smooth function $f$ on $X$, which equals
\[\begin{split}f&=\frac{(e^t\omega)\wedge \omega_{{\rm SRF}}^{n-1}\wedge\omega_B^m}{\omega_{{\rm SRF}}^n\wedge\omega_B^m}=\binom{n+m}{n}\frac{(e^t\omega)\wedge \omega_{{\rm SRF}}^{n-1}\wedge\omega_B^m}{\Omega}\\
&=e^{\vp+\dot{\vp}}\binom{n+m}{n}\frac{\omega\wedge (e^{-t}\omega_{{\rm SRF}})^{n-1}\wedge\omega_B^m}{\omega^{n+m}}.
\end{split}\]
We wish to show that $f$ converges to $1$ uniformly on $X$ and exponentially fast.
By the arithmetic-geometric means inequality we have
$$(f|_{X_y})^n\geq \frac{(e^t\omega|_{X_y})^n}{\omega_{{\rm SRF},y}^n},$$
and the right hand side converges uniformly to $1$ and exponentially fast by \eqref{1volcon}.
Therefore $1-f$ satisfies condition (c) in Lemma \ref{1conve}. Condition (b) is trivial, and condition (a) holds thanks to Lemma \ref{fiberC1}. Therefore we conclude that
\begin{equation} \label{2}
\|(e^t\omega|_{X_y})\wedge \omega_{{\rm SRF},y}^{n-1}-\omega_{{\rm SRF},y}^n\|_{C^0(X_y,\omega_0|_{X_y})}\leq Ce^{-\eta t},
\end{equation}
for some $C,\eta>0$ and all $y\in B$.

Choosing local fiber coordinates at a point  $x \in X_y$ so that $\omega_{\textrm{SRF},y}$ is the identity and $e^t \omega|_{X_y}$ is given by the positive definite $n \times n$ matrix $A$ we see that (\ref{1volcon}) and (\ref{2}) give $| \det A-1 | \le Ce^{-\eta t}$ and $| \textrm{tr} \, A - n | \le Ce^{-\eta t}$ respectively.  Applying Lemma \ref{lemmamatrix} we obtain $\| A - I \| \le Ce^{-\eta t/2}$. This implies the estimate \eqref{mtii}.
\end{proof}

We now turn to the proof of part (i) of Theorem \ref{mainthm1}.

\begin{proof}[Proof of Theorem \ref{mainthm1}.(i)]
Fix a point $x \in X$ with $y=\pi(x) \in B$.  We have
\begin{equation} \label{troto}
\tr{\omega}{\ti{\omega}}=\tr{\omega}{\omega_B}+\tr{\omega}{(e^{-t}\omega_{{\rm SRF},y})}+e^{-t}(-\tr{\omega}{\omega_B}+\tr{\omega}{(\of-\omega_{{\rm SRF},y})}).
\end{equation}
From Lemma \ref{lemmaknown1}, $\tr{\omega}{\omega_B}$ is uniformly bounded.  We claim that
\begin{equation} \label{trdiff}
| \tr{\omega}{(\of-\omega_{{\rm SRF},y})} | \le C e^{t/2}.
\end{equation}
Indeed, since $\pi$ is a submersion there are coordinates $z^1, \ldots, z^{m+n}$ near $x$ and $z^1, \ldots, z^m$ near $\pi(x)$ so that $\pi$ is given by the map $(z^1, \ldots, z^{m+n}) \mapsto (z^1, \ldots, z^m)$ (see e.g. \cite[p.60]{Ko}).  In particular, $z^{m+1}, \ldots, z^{m+n}$ restrict to local coordinates along the fibers of $\pi$.
Then we can write at $x$,
$$\of-\omega_{{\rm SRF},y} = 2\textrm{Re} \left( \sqrt{-1} \sum_{\alpha =1}^{m} \sum_{j=1}^{m+n} \Psi_{\alpha\ov{j}} dz^{\alpha} \wedge d\ov{z}^{j}\right),$$
for  $\Psi_{\alpha \ov{j}} \in \mathbb{C}$. We can do this because the form $\of-\omega_{{\rm SRF},y}$ vanishes when restricted to the fiber $X_y$.  Then
$$\left| \tr{\omega}{(\of-\omega_{{\rm SRF},y})} \right| = 2 \left| \textrm{Re} \left( \sum_{\alpha=1}^m \sum_{j=1}^{m+n} g^{\alpha \ov{j}} \Psi_{\alpha \ov{j}}\right)  \right| \le Ce^{t/2},$$
since thanks to estimate in Lemma \ref{lemmaknown1}.(i) we have that $|g^{j\ov{j}}|\leq Ce^t$ and $|g^{\alpha\ov{\alpha}}|\leq C$ whenever $1 \le \alpha \le m$, and so the Cauchy-Schwarz inequality gives $|g^{\alpha \ov{j}}| \le |g^{\alpha\ov{\alpha}} g^{j\ov{j}}|^{\frac{1}{2}}\leq  C e^{t/2}.$
Moreover, by compactness, we may assume that the constant $C$ is independent of the point $x$  and the choice of coordinates. Thus the claim (\ref{trdiff}) is proved.

Hence
\begin{equation}\label{1two}
|e^{-t}(-\tr{\omega}{\omega_B}+\tr{\omega}{(\of-\omega_{{\rm SRF},y})})|\leq Ce^{-t/2}.
\end{equation}
On the other hand, the estimate (\ref{mtii}) implies that
\begin{equation} \label{1three}
\tr{\omega}{(e^{-t}\omega_{{\rm SRF},y})}=\tr{(e^t\omega|_{X_y})}{\omega_{{\rm SRF},y}} \le n+ Ce^{-\eta t}.
\end{equation}
From (\ref{troto}), (\ref{1two}), (\ref{1three}) and Lemma \ref{1con}, we have
\begin{equation}\label{1one}
\tr{\omega}{\ti{\omega}}\leq n+m+Ce^{-\eta t},
\end{equation}
for some $\eta>0$.  Moreover, the constants $C$ and $\eta$ are independent of the choice of $x \in X$.
On the other hand, for $t \ge T_I$,
\begin{equation} \label{eqndetr}
\begin{split}
\frac{\ti{\omega}^{n+m}}{\omega^{n+m}} = {} & \frac{(e^{-t}\of+(1-e^{-t})\omega_B)^{n+m}}{\binom{n+m}{n}\of^n\wedge\omega_B^m} e^{nt}e^{-\vp-\dot{\vp}}  \\
\geq {} &  e^{-\vp-\dot{\vp}}- Ce^{-t} \\
\ge {} &  1- C'e^{-\eta t},
\end{split}
\end{equation}
for some $\eta>0$.
 Here we used Lemma \ref{lemmaknown1}.(iv) for the last inequality, and the fact that for $k\geq 1$ the term
$$\frac{e^{-(n+k)t}\of^{n+k}\wedge\omega_B^{m-k}}{\of^n\wedge\omega_B^m}e^{nt},$$
is of the order of $e^{-kt}$.

From (\ref{1one}) and (\ref{eqndetr}) we can apply Lemma \ref{lemmamatrix} (choose coordinates so that $\omega$ is the identity and $\tilde{\omega}$ is given by a matrix $A$) to obtain
$$\| \omega - \tilde{\omega} \|_{C^0(X, \omega)} \le C e^{-\eta t},$$
for a uniform $C, \eta>0$.  Since $\omega \le C \omega_0$, we have
$$\| \omega - \tilde{\omega} \|_{C^0(X, \omega_0)} \le C e^{-\eta t},$$
and since $\tilde{\omega}= \omega_B + e^{-t}(\of - \omega_B)$ we obtain the estimate (\ref{mti}) as required.
\end{proof}

Finally, part (iii) of Theorem \ref{mainthm1} follows from part (i) and the definition of Gromov-Hausdorff convergence, for example using \cite[Lemma 9.1]{TWY} (note that our submersion $\pi:X \rightarrow B$ is a smooth fiber bundle).

\section{The K\"ahler-Ricci flow in the case of singular fibers} \label{sectionkrfsing}

In this section we give the proof of Theorem \ref{mainthm2}.

\subsection{Preliminaries: the case of singular fibers} \label{sectionprelim2}
Let us first recall the general setup of Song-Tian \cite{ST,ST2,ST3} where our results will apply.
Let $(X^{m+n},\omega_0)$ be a compact K\"ahler manifold with canonical bundle $K_X$ semiample and $0<m:=\kappa(X) < \dim X$.

The map $\pi: X \rightarrow B$ is a surjective holomorphic map given by sections of $H^0(X, K_X^{\ell})$ where $B^m$ is a normal projective variety and the generic fiber
 $X_y=\pi^{-1}(y)$ of $\pi$ has $K_{X_y}^\ell$ holomorphically trivial (so it is in particular a Calabi-Yau manifold of dimension $n$).
Recall that we denote by $S'\subset B$ the singular set of $B$ together with the set of critical values of $\pi$, and we define $S=\pi^{-1}(S') \subset X$.

Since the map $\pi:X\to B\subset \mathbb{P}H^0(X,K_X^\ell)$ is induced by the space of global sections of $K_X^\ell$, we have that $\pi^*\mathcal{O}(1)=K_X^\ell$. Therefore, if we let $\chi$ be $\frac{1}{\ell}\omega_{{\rm FS}}$ on $\mathbb{P}H^0(X,K_X^\ell)$, we have that $\pi^*\chi$  (later, denoted by $\chi$) is a smooth semipositive representative of $-c_1(X)$.  Here, $\omega_{\textrm{FS}}$ denotes the Fubini-Study metric.  We will also denote by $\chi$ the restriction of $\chi$ to $B\backslash S'$.

As in the submersion case, we can define the semi Ricci-flat form $\of$ on $X\backslash S$, so that $\omega_{{\rm SRF},y}=\of|_{X_y}$ is the unique Ricci-flat K\"ahler metric on $X_y$ cohomologous to $\omega_0|_{X_y}$, for all $y\in B\backslash S'$.
Let $\Omega_1$ be the volume form on $X$ with
$$\ddbar \log \Omega_1 =\chi, \quad \int_X \Omega_1=\binom{n+m}{n}\int_X \omega_0^n\wedge\chi^m.$$
Define a function $F$ on $X\backslash S$ by
\begin{equation} \label{eqnvv}
F = \frac{\Omega_1}{\binom{n+m}{n} \omega_{\textrm{SRF}}^n \wedge \chi^m}.
\end{equation}
As in the proof of Theorem \ref{genke}, one sees easily that $F$ is constant along the fibers $X_y$, $y\in B\backslash S'$, so it descends to a smooth function $F$ on $B\backslash S'$.
Then \cite[Section 3]{ST2} shows that the Monge-Amp\`ere equation
\begin{equation} \label{eqnv}
(\chi+\ddbar v)^m=Fe^v\chi^m,
\end{equation}
has a unique solution $v$ (in the sense of Bedford-Taylor \cite{BT}) which is a bounded $\chi$-plurisubharmonic function on $B$, smooth on $B\backslash S'$. The $L^\infty$ bound for $v$ uses the pluripotential estimates of Ko\l odziej \cite{Kol} (see also the survey \cite{PSoS}). We define $$\omega_B:=\chi+\ddbar v,$$ which is a smooth K\"ahler metric on $B\backslash S'$, and satisfies the twisted K\"ahler-Einstein equation
$$\Ric(\omega_B)=-\omega_B+\omega_{{\rm WP}},$$
where $\omega_{{\rm WP}}$ is the smooth semipositive Weil-Petersson form on $B\backslash S'$  constructed  in Section \ref{subsectionsub}.
We also define
$$\Omega=\binom{n+m}{n}\of^n\wedge \omega_B^m,$$
and using \eqref{eqnvv} and \eqref{eqnv} we observe that
$$\Omega=\Omega_1\frac{\omega_B^m}{F\chi^m}=\Omega_1e^v,$$
which is thus a bounded strictly positive volume form on $X$, smooth on $X\backslash S$. As in the proof of Theorem \ref{genke}, we can see that
$$\ddbar \log \Omega = \omega_B,$$
holds on $X\backslash S$.

Let then $\omega=\omega(t)$ be a solution of the K\"ahler-Ricci flow
$$\frac{\de}{\de t}\omega=-\Ric(\omega)-\omega,\quad \omega(0)=\omega_0,$$
which exists for all time.
We gather together some facts about the K\"ahler-Ricci flow in this case which are already known, or which follow easily from the existing literature.

We begin by writing the flow as a parabolic complex Monge-Amp\`ere equation.  Since $\Omega$ is not smooth everywhere, it is more convenient to use the smooth volume form $\Omega_1$. Similarly, since the limiting metric $\omega_B$ is not smooth everywhere, we replace it by $\chi$ which is smooth on the whole of $X$.

We therefore define the reference metrics
$$\hat{\omega}=e^{-t}\omega_0+(1-e^{-t})\chi,$$
which are K\"ahler for all $t\geq 0$, and for all $t\geq 0$ we can write $\omega=\hat{\omega}+\ddbar\vp$, and $\vp(0)=0$.
Then the K\"ahler-Ricci flow is equivalent to the parabolic complex Monge-Amp\`ere equation
\begin{equation} \label{pcma2}
\frac{\de}{\de t}\vp=\log\frac{e^{nt}(\hat{\omega}+\ddbar\vp)^{n+m}}{\Omega_1}-\vp, \quad \vp(0)=0,\quad \hat{\omega}+\ddbar\vp>0.
\end{equation}
It is convenient to introduce, following \cite[Section 2]{To}, a smooth nonnegative function $\sigma$ on $X$ with zero locus exactly equal to $S$ and with
\begin{equation}\label{stupidbounds}
\sigma\leq 1, \quad \mn\de\sigma\wedge\db\sigma\leq C\chi,\quad -C\chi\leq \ddbar\sigma\leq C\chi,
\end{equation}
for some constant $C$ (in the case when $S$ is empty, i.e. when $\pi$ is a submersion and $B$ nonsingular, we can set $\sigma=1$).
Explicitly, let $\mathcal{I}$ be the ideal sheaf of $S'$ inside $\mathbb{P}H^0(X,K_X^\ell)$, let $\{U_j\}$ be an open cover of $\mathbb{P}H^0(X,K_X^\ell)$ such that
on each $U_j$ the ideal $\mathcal{I}$ is generated by finitely many holomorphic functions $\{f_{j,k}\}$, let $\rho_j$ be a partition of unity subordinate to this cover, and define
a smooth function on $\mathbb{P}H^0(X,K_X^\ell)$ by
\begin{equation}\label{constrsigma}
\sigma=C^{-1}\sum_{j,k}\rho_j|f_{j,k}|^2,
\end{equation}
where $C$ is a constant chosen so that $\sigma\leq 1$. Then the pullback of $\sigma$ to $X$ is the function that we need.
We have the following lemma, which is an analogue of Lemma \ref{lemmaknown1}.

\begin{lemma} \label{lemmape2}
Let $\omega=\omega(t)$ solve the K\"ahler-Ricci flow as above, and write $\varphi=\varphi(t)$ for the solution of (\ref{pcma2}).  Then
\begin{enumerate}
\item[(i)] For any compact set $K\subset X\backslash S$ there is a constant $C=C(K)$ such that
$C^{-1}\hat{\omega}\leq \omega\leq C\hat{\omega}$
 on $K$ for all $t\geq 0$.
\item[(ii)] There exists $\lambda>0$ and a positive decreasing function $h(t)$ which tends to zero as $t \rightarrow \infty$ such that
\begin{equation} \label{vpv}
\sup_X |\sigma^\lambda (\vp-v)|\leq h(t), \quad \textrm{for all }t \ge 0.
\end{equation}
\item[(iii)] There exists $C$ such that $|R| \le C$ on $X \times [0,\infty)$.
\item[(iv)]  There exists $C$ such that, for $h$,  $\lambda$ as in (ii),
\begin{equation}\label{dot}
\sup_X |\sigma^\lambda (\vp+\dot{\vp}-v)|\leq Ch(t)^{\frac{1}{2}}, \quad \textrm{for all } t \ge 0.
\end{equation}
\end{enumerate}
\end{lemma}
\begin{proof}
Part (i) is proved in \cite{FZ} (and is a direct adaptation of \cite{To}, see also \cite{ST} for the case of elliptic surfaces).
Part (ii) is due to Song-Tian (see the proof of \cite[Proposition 5.4]{ST2}).  Note that we are free to increase the value of $\lambda$, at the expense of changing the function $h(t)$.  Part (iii) was proved in \cite{ST3}.

The proof of (iv) is almost identical to the proof of Lemma \ref{lemmaknown1}.(iv).  Indeed, it is enough to show that
$\sup_X |\sigma^\lambda \dot{\vp}|\leq Ch(t)^{\frac{1}{2}}$.
We have as before $\ddt{} \dot{\varphi} = -R - m -\dot{\varphi}$ and hence $|\dot{\varphi}|,  |\partial \dot{\varphi}/\partial t| \le C_0$ for some uniform constant $C_0$.
Suppose for a contradiction that we do not have the desired upper bound of $\sigma^{\lambda}\dot{\varphi}$.  Then there exists a sequence $(x_k, t_k) \in X \times [0,\infty)$ with $t_k \rightarrow \infty$ as $k \rightarrow \infty$ such that
$$\sigma^{\lambda}(x_k)\dot{\varphi}(x_k, t_k) \ge k h(t_k)^{\frac{1}{2}}.$$
In particular, $x_k\not\in S.$
Put $\gamma_k = \frac{k}{2C_0\sigma^{\lambda}(x_k)} h(t_k)^{\frac{1}{2}}$.  At $x_k$ we have
$\sigma^{\lambda}\dot{\varphi} \ge \frac{k}{2} h(t_k)^{\frac{1}{2}}$ on $[t_k, t_k + \gamma_k ].$
Then, using \eqref{vpv}, we have at $x_k$,
\[\begin{split}
2 h(t_k) \ge (h(t_k+\gamma_k)+h(t_k)) & \ge \sigma^{\lambda}\int_{t_k}^{t_k+\gamma_k} \dot{\varphi} dt \\
&\ge \gamma_k \frac{k}{2} h(t_k)^{\frac{1}{2}} =  \frac{k^2}{4 C_0\sigma^{\lambda}} h(t_k),
\end{split}\]
which gives a contradiction when $k \rightarrow \infty$ and we are done.
The lower bound is similar.
\end{proof}

The next lemma follows from a straightforward computation.

\begin{lemma}\label{calcola}
Along the K\"ahler-Ricci flow, we have on $X \backslash S$,
\begin{equation} \label{pdpv}
\left(\frac{\de}{\de t}-\Delta\right)(\vp+\dot{\vp}-v)=\tr{\omega}{\omega_B}-m.
\end{equation}
\end{lemma}
\begin{proof}
This follows from the evolution equations
$$\left(\frac{\de}{\de t}-\Delta\right)\vp=\dot{\vp}-(n+m)+\tr{\omega}{\hat{\omega}},$$
$$\left(\frac{\de}{\de t}-\Delta\right)\dot{\vp}=\tr{\omega}{(\chi-\hat{\omega})}+n-\dot{\vp},$$
and the equation $\Delta v= \tr{\omega}{(\omega_B - \chi)}$.
\end{proof}

Next, from  \cite{ST, ST2, ST3}, we have:

\begin{lemma} \label{lemmaps2}
There exist positive constants $C, C'$ and $\lambda$ such that
\begin{equation} \label{ps1}
\chi \le C \omega, \quad \omega_B \le C\sigma^{-\lambda} \chi, \quad \omega_B \le C \sigma^{-\lambda} \omega,
\end{equation}
wherever these quantities are defined, and
\begin{equation} \label{sigmainequalities}
|\nabla \sigma|^2\leq C, \quad |\Delta \sigma|\leq C.
\end{equation}
On $X \backslash S$ we have
\begin{equation} \label{ps2}
\left(\frac{\de}{\de t}-\Delta\right)\tr{\omega}{\omega_B}\leq \tr{\omega}{\omega_B}+C\sigma^{-\lambda}(\tr{\omega}{\omega_B})^2 \le C' \sigma^{-3\lambda}.
\end{equation}
\end{lemma}
\begin{proof}
The first inequality of (\ref{ps1}) is given in  \cite[Proposition 2.2]{ST3}.  The second is a consequence of the arguments in \cite[Theorem 3.3]{ST2}, as follows: Song-Tian construct the metric $\omega_B$ by working on a resolution $\mu:Y\to B$ so that $\mu^{-1}(S')$ is a simple normal crossing divisor $E$. The class $[\mu^*\chi]$ is semipositive and big, and since $\mu^*\chi=\frac{1}{\ell}\ti{\mu}^*\omega_{\rm FS}$, where $\ti{\mu}$ is the composition $Y\to B\to \mathbb{P}H^0(X,K_X^\ell)$, a standard argument shows that $[\mu^*\chi]-\ve [E]$ contains a K\"ahler metric $\omega_Y$, for some $\ve>0$. Then the results of \cite{EGZ} give a bounded solution $\hat{v}$ of the Monge-Amp\`ere equation
\begin{equation} \label{eqnv2}
(\mu^*\chi+\ddbar \hat{v})^m=(\mu^*F)e^{\hat{v}}(\mu^*\chi)^m,\quad \mu^*\chi+\ddbar \hat{v}>0,
\end{equation}
which is smooth away from $E$. If we let $\hat{\omega}_B=\mu^*\chi+\ddbar \hat{v}$, then $\hat{\omega}_B$ descends to the metric $\omega_B$ on $B\backslash S'$. Then \cite[Theorem 3.3]{ST2} gives
$$\hat{\omega}_B\leq C |s_E|^{-\alpha}_{h_E}\omega_Y,$$
for some $\alpha>0$, where $s_E$ is a defining section of the line bundle associated to $E$ and $h_E$ is a smooth metric on it. From the construction of $\sigma$ in \eqref{constrsigma}, it follows that $\mu^*\sigma$ is bounded above by $|s_E|_{h_E}^\beta$ for some $\beta>0$. Since $\mu:Y\backslash E\to B\backslash S'$ is an isomorphism, the metric $\omega_Y$ induces a metric $\omega_Y'$ on $B\backslash S'$ and we have
\begin{equation}\label{bs1}
\omega_B\leq C \sigma^{-\gamma}\omega_Y',
\end{equation}
for some $\gamma>0$. But on $Y$ we also have
$$\omega_Y\leq  C|s_E|_{h_E}^{-\delta} \mu^*\chi,$$
for some $\delta>0$, since the subvariety of $Y$ where $\ti{\mu}$ fails to be an immersion is contained in $E$. On $B\backslash S'$ this implies that
\begin{equation}\label{bs2}
\omega_Y' \leq C\sigma^{-\gamma'} \chi,
\end{equation}
for some $\gamma'>0$. Combining \eqref{bs1} and \eqref{bs2} we obtain the second inequality of \eqref{ps1}.
The third is obtained by combining the first and second inequalities.  The inequalities (\ref{sigmainequalities}) follow from (\ref{stupidbounds}) and the inequality $\chi \le C \omega$.

Local higher order estimates for the equation (\ref{eqnv}) imply that the bisectional curvature of $\omega_B$ is bounded by $C \sigma^{-\lambda}$ on $B\backslash S'$, up to increasing $\lambda$.  The Schwarz Lemma calculation \cite[Section 4]{ST} (see also the exposition in \cite[Theorem 3.2.6]{SW}) then gives the first inequality  (\ref{ps2}), and the second follows from (\ref{ps1}).
\end{proof}

\subsection{Collapsing estimates for the K\"ahler-Ricci flow away from the singular set}

We now give the proof of Theorem \ref{mainthm2}.  First, we have an analogue of Lemma \ref{1con}.  It is more complicated in this case because of the singular set, and the fact that the quantity $\sigma^{\lambda}(\varphi+\dot{\varphi} - v)$, which now takes the role of $\varphi + \dot{\varphi}$, does not decay exponentially.

\begin{lemma}\label{con22}
Given a compact set $K\subset X\backslash S$, there is a positive decreasing function $F(t)$ which goes to zero as $t\to\infty$ such that
$$\sup_K(\tr{\omega}{\omega_B}-m)\leq F(t).$$
\end{lemma}
\begin{proof}
Let $h(t)$ be a positive decreasing function which goes to zero as $t\to\infty$ such that
\begin{equation} \label{hcon}
\sup_X|\sigma^\lambda(\vp+\dot{\vp}-v)|\leq h(t),
\end{equation}
which exists thanks to \eqref{dot}. We may assume without loss of generality that  $h'(t)\to 0$ as $t\to\infty$.
Indeed, by assumption there exists a sequence $t_i\to\infty$ with $h(t)\leq \frac{1}{i}$ for all $t\geq t_i$, and $t_{i+1}-t_{i}\geq 1$.
Define a piecewise constant function $\ti{h}(t)$, $t\geq 0$, to be equal to $\frac{1}{i}$ on the interval $t_i\leq t< t_{i+1}$, so clearly $h(t)\leq \ti{h}(t)$ for all $t\geq t_1$. We can then
smooth out $\ti{h}$ in the obvious way (making it continuous and essentially linear on each interval $t_i\leq t< t_{i+1}$) to obtain a smooth nonnegative and decreasing function $\hat{h}(t)$ with $\ti{h}(t)\leq \hat{h}(t)$ for all $t\geq 0$ and so that $\hat{h}(t)$ still goes to zero and the derivative of $\hat{h}(t)$ on the interval $t_i\leq t< t_{i+1}$ is of the order of $(i(t_{i+1}-t_i))^{-1}$ which goes to zero. We then replace $h(t)$ by $\hat{h}(t)$ (still calling it $h(t)$).

Next, we pick a smooth positive function $\ell(t), t\geq 0$, with $\lim_{t\to\infty}\ell(t)=\infty$, such that $|\ell'(t)|\leq C$ for all $t\geq 0$ and
$$\ell(t)\leq\frac{1}{h(t)^{\frac{1}{2}}},\quad \ell(t)\leq -\frac{1}{h'(t)}.$$
Define
$$Q=\ell(t)\sigma^{3\lambda}(\tr{\omega}{\omega_B}-m)-\frac{1}{h(t)}\sigma^{3\lambda}(\vp+\dot{\vp}-v).$$
Then using Lemmas \ref{calcola}, \ref{lemmaps2} and (\ref{hcon}), we have on $X \backslash S$,
\[\begin{split}
\lefteqn{\left(\frac{\de}{\de t}-\Delta\right)Q} \\&\leq \left(\ell'(t)-\frac{1}{h(t)}\right)\sigma^{3\lambda}(\tr{\omega}{\omega_B}-m)+C(\ell(t)+1)+\frac{h'(t)}{h(t)^2}\sigma^{3\lambda}(\vp+\dot{\vp}-v)\\
&-2\mathrm{Re}\left\langle \nabla(\sigma^{3\lambda}),\nabla\left(\ell(t) (\tr{\omega}{\omega_B}-m)-\frac{1}{h(t)}(\vp+\dot{\vp}-v)\right)\right\rangle\\
&\le \left(\ell'(t)-\frac{1}{h(t)}\right)\sigma^{3\lambda}(\tr{\omega}{\omega_B}-m)+C(\ell(t)+1)- \frac{h'(t)}{h(t)}\\
&-2\sigma^{-3\lambda}\mathrm{Re}\left\langle \nabla(\sigma^{3\lambda}),\nabla Q\right\rangle+C Q\sigma^{-2},
\end{split}\]
where we have used the fact that $|\Delta ( \sigma^{3\lambda} ) | \le C \sigma^{\lambda}$.
Observe that from (\ref{ps1}) and (\ref{hcon}) we have $\sigma^{-2}Q\leq C\ell(t)$, since we may assume that $\lambda\geq 1$. Since $|\ell'(t)|\le C$,   we may assume that $t$ is sufficiently large so that
$$\ell'(t)\leq \frac{1}{2h(t)}.$$
We wish to obtain an upper bound for $Q$ using the maximum principle.  Hence we may assume without loss of generality that we are working at a point with $\tr{\omega}{\omega_B}-m>0$.
Therefore we get
\[
\begin{split}
\left(\frac{\de}{\de t}-\Delta\right)Q
 \leq {} &  -\frac{1}{2h(t)}\sigma^{3\lambda}(\tr{\omega}{\omega_B}-m)+C(\ell(t)+1)-\frac{h'(t)}{h(t)}\\ & -2\sigma^{-3\lambda}\mathrm{Re}\left\langle \nabla(\sigma^{3\lambda}),\nabla Q\right\rangle,
\end{split}
\]
and so at a maximum point of $Q$ (which is necessarily in $X\backslash S$) we have
$$Q\leq C\ell(t)(\ell(t)+1)h(t)-2\ell(t)h'(t)+C\leq C,$$
thanks to our choice of $\ell(t)$. This proves what we want, choosing $F(t)=\frac{C}{\ell(t)},$ for $C$ depending on the compact set $K$.
\end{proof}

The rest of the proof of Theorem \ref{mainthm2} is almost identical to the proof of Theorem \ref{mainthm1}.

\begin{proof}[Proof of Theorem \ref{mainthm2}]
To avoid repetition, we provide here just an outline of the proof, emphasizing  the changes from the proof of Theorem \ref{mainthm1}.

For the rest of the proof,  fix
a compact set $K'\subset B\backslash S'$ and write $K=\pi^{-1}(K') \subset X \backslash S$.

First, there exists a constant $C$ such that
$$\|e^t\omega|_{X_y}\|_{C^1(X_y,\omega_0|_{X_y})}\leq C , \quad e^t\omega|_{X_y}\geq C^{-1}\omega_0|_{X_y},$$
for all $t\geq 0$ and for all $y\in K'$.   Indeed, this follows from the same argument as in Lemma \ref{fiberC1}, since $\pi$ is a submersion near every point in $\pi^{-1}(K')$.

In particular,  for any fixed $y\in B\backslash S'$ and $0<\alpha<1$, given any sequence $t_k\to\infty$ we can extract a subsequence such that
\begin{equation} \label{Calphafirst}
e^{t_k}\omega|_{X_y}\to\omega_{\infty,y}, \quad \textrm{in $C^\alpha$ on $X_y$},
\end{equation}
where $\omega_{\infty,y}$ is a $C^\alpha$ K\"ahler metric on $X_y$ cohomologous to $\omega_0|_{X_y}$.

Next we claim that we have fiberwise $C^{\alpha}$ convergence of $e^t \omega|_{X_y}$ to $\omega_{\textrm{SRF},y}$ for $y \in K'$.  More precisely,   given $\alpha \in (0,1)$ and  $y\in K'$,  we have that
\begin{equation} \label{fiberwiseCalpha}
e^{t}\omega|_{X_y}\to\omega_{{\rm SRF},y}, \quad \textrm{in $C^\alpha$ as $t\to \infty$  on $X_y$.}
\end{equation}
Indeed, on $X\backslash S$ write
$$f = e^{\vp+\dot{\vp}-v}\binom{n+m}{n}\frac{\omega^n\wedge\omega_B^{m}}{\omega^{n+m}},$$
so that restricted to $X_y$ ($y \in K'$) we have $(e^t \omega|_{X_y})^n = f ( \omega_{\textrm{SRF},y})^n$.
Note that
$$\int_{X_y} (f-1) ( \omega_{\textrm{SRF},y})^n = 0,$$
and, thanks to \eqref{dot},
$$\sup_{K} | e^{\varphi + \dot{\varphi} - v} -1 | \rightarrow 0, \quad \textrm{as } t\rightarrow \infty.$$
Moreover, as in (\ref{1amgm}) and using Lemma \ref{con22}, on $K$ we have
$$\binom{n+m}{n}\frac{\omega^n\wedge\omega_B^{m}}{\omega^{n+m}}\le (\tr{\omega}{\omega_B}/m)^{m} \leq 1+ H(t),$$
for $H(t) \rightarrow 0$ as $t\rightarrow \infty$.  Hence $f-1 \le h(t),$ with $h(t)\rightarrow 0$ as $t \rightarrow \infty$.  Finally, $f$ satisfies
$| \nabla (f|_{X_y})|_{\omega_0|_{X_y}} \le A$ for all $y \in K'$.  Applying Lemma \ref{1conve} to $f-1$ on $K$ (see Remark \ref{remarkcpt}) we have that $f$ converges to $1$ uniformly on $K$.  Namely,
\begin{equation}\label{volcon2}
\|(e^{t}\omega|_{X_y})^n-\omega_{{\rm SRF},y}^n\|_{C^0(X_y,\omega_0|_{X_y})}\to 0, \quad \textrm{as } t\rightarrow \infty,
\end{equation}
 \emph{uniformly} as $y$ varies in $K'$.  In particular, $e^{t}\omega|_{X_y}\to\omega_{{\rm SRF},y}$ in $C^\alpha$ by (\ref{Calphafirst}) and
the same argument as in Lemma \ref{1conv}.

Next we show that
\begin{equation} \label{fwca2}
\| e^t \omega|_{X_y} - \omega_{\textrm{SRF},y} \|_{C^0(X_y, \omega_0|_{X_y})} \rightarrow 0, \quad \textrm{as } t \rightarrow \infty,
\end{equation}
uniformly for $y \in K'$.  To see this, define
$$f = e^{\vp+\dot{\vp}-v}\binom{n+m}{n}\frac{\omega\wedge (e^{-t}\omega_{{\rm SRF}})^{n-1}\wedge\omega_B^m}{\omega^{n+m}},$$
which satisfies
$$f|_{X_y}=\frac{(e^t\omega|_{X_y})\wedge \omega_{{\rm SRF},y}^{n-1}}{\omega_{{\rm SRF},y}^n},$$
when restricted to $X_y$, for $y \in K'$.  By the same argument as in the proof of (\ref{mtii}) in Section 2, we see that $f$ converges 1 as $t \rightarrow 0$ and hence
\begin{equation} \label{g22}
\| (e^t \omega|_{X_y}) \wedge \omega_{\textrm{SRF},y}^{n-1} - \omega_{\textrm{SRF},y}^n \|_{C^0(X_y, \omega_0|_{X_y})} \rightarrow 0, \quad \textrm{as } t \rightarrow \infty,
\end{equation}
uniformly as $y$ varies in $K'$.  From (\ref{volcon2}) and (\ref{g22}), we apply  Lemma \ref{lemmamatrix} to obtain (\ref{fwca2}) as required.

It remains to prove part (i) of Theorem \ref{mainthm2}.   Define
$$\tilde{\omega} = e^{-t} \of + (1-e^{-t}) \omega_B,$$
as in Section \ref{sectionkrfsubmersion}.  On $K$ we have
$$\tr{\omega}{\ti{\omega}}=\tr{\omega}{\omega_B}+\tr{\omega}{(e^{-t}\omega_{{\rm SRF},y})}+e^{-t}(-\tr{\omega}{\omega_B}+\tr{\omega}{(\of-\omega_{{\rm SRF},y})}).$$
On $K$, $\tr{\omega}{\omega_B}$ is uniformly bounded.  It then follows by the same argument as in the proof of (\ref{trdiff}) that
\begin{equation}\label{two}
\sup_K|e^{-t}(-\tr{\omega}{\omega_B}+\tr{\omega}{(\of-\omega_{{\rm SRF},y})})|\leq Ce^{-t/2}.
\end{equation}
Moreover, (\ref{fwca2}) implies that $\tr{\omega}{(e^{-t} \omega_{\textrm{SRF},y})} \le n + h(t)$ for $h(t) \rightarrow 0$, on $K$.
Then from Lemma \ref{con22} we have on $K$, after possibly changing $h(t)$,
\begin{equation}\label{one}
\tr{\omega}{\ti{\omega}}-(n+m)\leq h(t) \rightarrow 0.
\end{equation}
On the other hand, again on $K$,
$$\frac{\ti{\omega}^{n+m}}{\omega^{n+m}}=\binom{n+m}{n}\frac{e^{-nt}\of^n\wedge\omega_B^m}{\omega^{n+m}}+O(e^{-t})=e^{-\vp-\dot{\vp}+v}+O(e^{-t})\to 1.$$
Applying Lemma \ref{lemmamatrix}, we obtain
$$\| \omega - \tilde{\omega} \|_{C^0(K, \omega)} \rightarrow 0, \quad \textrm{as } t \rightarrow \infty.$$  But since $\omega \le C \omega_0$ on $K$, and $\tilde{\omega} = \omega_B + e^{-t}(\of - \omega_B)$ we obtain
$$\| \omega - \omega_B \|_{C^0(K, \omega_0)} \rightarrow 0,  \quad \textrm{as } t \rightarrow \infty,$$
as required.
\end{proof}

\section{Collapsing of Ricci-flat K\"ahler metrics}\label{sectioncy}
In this section we prove Theorem \ref{mainthm3}.

\subsection{Monge-Amp\`ere equations and preliminary estimates}

Let $(X, \omega_X)$ be a compact $(n+m)$ manifold with a Ricci-flat K\"ahler metric $\omega_X$, as in Section \ref{introcy} of the Introduction.  Recall that we have a holomorphic map $\pi: (X, \omega_X)  \rightarrow (Z, \omega_Z)$ between compact K\"ahler manifolds with (possibly singular) image $B$, a normal variety in $Z$.  We denote by $S' \subset B$ those points of $B$ which are either singular in $B$ or critical for $\pi$, and we write $S=\pi^{-1}(S')$ .  The fibers $X_y$ for $y \in B \backslash S'$ are Calabi-Yau $n$-folds.

Write $\chi=\pi^*\omega_Z$, which is a smooth nonnegative $(1,1)$ form on $X$, and we will also write $\chi$ for the restriction of $\omega_Z$ to $B\backslash S'$.  Note that
$\int_{B\backslash S'}\chi^m$ is finite.

We define a semi Ricci-flat form $\of$ on $X\backslash S$ in the same way as in Section \ref{subsectionsub}.  Indeed, for each $y \in B \backslash S'$ there is a smooth function $\rho_y$ on $X_y$ so that $\omega_X|_{X_y} + \ddbar \rho_y = \omega_{\textrm{SRF,y}}$ is Ricci-flat, normalized by $\int_{X_y} \rho_y (\omega_X|_{X_y})^n=0$.  As $y$ varies, this defines a smooth function $\rho$ on $X\backslash S$ and we define
$\of = \omega_X + \ddbar \rho.$

Let $F$ be the function on $X\backslash S$ given by
$$F=\frac{\omega_X^{n+m}}{\binom{n+m}{n}\of^n\wedge\chi^m}.$$
As in the proof of Theorem \ref{genke}, one sees easily that $F$ is constant along the fibers $X_y$, $y\in B\backslash S'$, so it descends to a smooth function $F$ on $B\backslash S'$.  We see that $F$ satisfies $\int_{B\backslash S'}F\chi^m=\int_X\omega_X^{n+m}/\binom{n+m}{n}\int_{X_y}\omega_X^n$ (see \cite[Section 3]{ST2} and \cite[Section 4]{To}). Here note that $\int_{X_y}\omega_X^n$ is independent of $y\in B\backslash S'$.

Then \cite[Section 3]{ST2} shows that the Monge-Amp\`ere equation
\begin{equation} \label{eqnvcy}
(\chi+\ddbar v)^m=\frac{\binom{n+m}{n}\int_{X} \omega_X^n\wedge\chi^m}{\int_{X}\omega_X^{n+m}}F\chi^m,
\end{equation}
has a unique solution $v$ (in the sense of Bedford-Taylor \cite{BT}) which is a bounded $\chi$-plurisubharmonic function on $B$, smooth on $B\backslash S'$, with $\int_{X}v\omega_X^{n+m}=0$, where here and henceforth we write $v$ for $\pi^*v$.

We define
$$\omega_B=\chi+\ddbar v,$$
for $v$ solving (\ref{eqnvcy}).  Note that we have
\begin{equation}\label{ct3}
\of^n\wedge\omega_B^m=\frac{\binom{n+m}{n}\int_{X} \omega_X^n\wedge\chi^m}{\int_{X}\omega_X^{n+m}} F\of^n\wedge\chi^m=\frac{\int_{X} \omega_X^n\wedge\chi^m}{\int_{X}\omega_X^{n+m}}\omega_X^{n+m}.
\end{equation}
Moreover, $\omega_B$ is a smooth K\"ahler metric on $B\backslash S'$, and satisfies
$$\Ric(\omega_B)=\omega_{{\rm WP}},$$
where $\omega_{{\rm WP}}$ is a smooth semipositive form on $B\backslash S'$ constructed in the same way as in Section \ref{subsectionsub}.

As in Section \ref{sectionprelim2}, we fix a smooth nonnegative function $\sigma$ on $X$ with zero locus exactly equal to $S$ and with
\begin{equation}\label{stupidbounds2}
\sigma\leq 1, \quad \mn\de\sigma\wedge\db\sigma\leq C\chi,\quad -C\chi\leq \ddbar\sigma\leq C\chi,
\end{equation}
for some constant $C$ (in the case when $S$ is empty we set $\sigma=1$).  It is convenient to define another smooth nonnegative function $\mathcal{F}$ with zero locus equal to $S$ by
\begin{equation} \label{defnF}
\mathcal{F} = e^{-e^{A\sigma^{-\lambda}}},
\end{equation}
for positive constants $A$ and $\lambda$ to be determined.

For $t\geq 0$, let $\tilde{\omega}=\tilde{\omega}(t)$, given by
$$\ti{\omega}=\chi+e^{-t}\omega_X \in \alpha_t = [\chi] + e^{-t} [\omega_X],$$ be a family of reference K\"ahler metrics, and let $\omega=\ti{\omega}+\ddbar\vp$ be the unique Ricci-flat K\"ahler metric on $X$ cohomologous to $\ti{\omega}$, with the normalization $\int_X\vp\omega_X^{n+m}=0$. Then $\omega$ solves the Calabi-Yau equation
\begin{equation}\label{ma}
\omega^{n+m}=c_t e^{-nt}\omega_X^{n+m},
\end{equation}
where $c_t$ is the constant given by
\begin{equation}\label{ct}
c_t=\frac{\int_X e^{nt}\ti{\omega}^{n+m}}{\int_X\omega_X^{n+m}}=\frac{1}{\int_X\omega_X^{n+m}}\sum_{k=0}^m\binom{n+m}{k}e^{-(m-k)t}\int_X \omega_X^{n+m-k}\wedge\chi^k,
\end{equation}
so
\begin{equation} \label{ct2}
c_t=\binom{n+m}{n} \frac{\int_X\omega_X^{n}\wedge\chi^m }{\int_X\omega_X^{n+m}}+O(e^{-t}).
\end{equation}

The following estimates are already known by the work of the first-named author \cite{To}.

\begin{lemma} \label{lemmape3}
 Let $\omega=\omega(t)$ and $\varphi=\varphi(t)$ be as above.  Then
\begin{enumerate}
\item[(i)]  There exist positive constants $C, A, \lambda$ such that for $\mathcal{F}$ given by (\ref{defnF}),
\begin{equation}\label{bounds}
C^{-1} \mathcal{F} \tilde{\omega} \le \omega \le C \mathcal{F}^{-1} \tilde{\omega}, \quad \textrm{for all } t\ge 0.
\end{equation}
\item[(ii)] There exists $C$ such that $\sup_X | \varphi| \le C$ for all $t \ge 0$.
\item[(iii)] $\displaystyle{ \int_X | \varphi - v | \, \omega_X^{n+m}  \rightarrow 0}$ as $ t\rightarrow \infty$.
\item[(iv)] For any $\alpha \in (0,1)$, we have $\varphi \rightarrow v$ in $C^{1, \alpha}_{\emph{loc}}(X \setminus S)$ as $t \rightarrow \infty$.
\end{enumerate}
\end{lemma}
\begin{proof}
Part (i) is proved in \cite{To} (see also \cite[Lemma 4.1]{GTZ}).  For part (ii), see \cite[Theorem 2.1]{To}.  Parts (iii) and (iv) are proved in \cite[Theorem 4.1]{To}.
\end{proof}

The estimate \eqref{bounds} is the reason why here and in the rest of this section we consider the function $\F$ instead of the simpler $\sigma^{-\lambda}$ which we used in the case of the K\"ahler-Ricci flow. We do not expect this to be optimal.

The next lemma is analogous to Lemma \ref{lemmaps2} above.

\begin{lemma} \label{lemmaps3}
There exist positive constants $C, C'$ and $\lambda$ such that
\begin{equation} \label{ps3}
\chi \le C \omega, \quad \omega_B \le C\sigma^{-\lambda} \chi, \quad \omega_B \le C \sigma^{-\lambda} \omega,
\end{equation}
wherever these quantities are defined, and
\begin{equation} \label{sigmainequalities2}
|\nabla \sigma|^2\leq C, \quad |\Delta \sigma|\leq C.
\end{equation}
On $X \backslash S$ we have
\begin{equation} \label{ps4}
\Delta \tr{\omega}{\omega_B}\geq - \tr{\omega}{\omega_B}-C\sigma^{-\lambda}(\tr{\omega}{\omega_B})^2 \ge -C' \sigma^{-3\lambda}.
\end{equation}
\end{lemma}
\begin{proof}  The proof is almost the same as that of Lemma \ref{lemmaps2}.
The first inequality of (\ref{ps3}) is given by the Schwarz Lemma \cite[Lemma 3.1]{To}.  The second follows again from the arguments in \cite[Theorem 3.3]{ST2}, as in
Lemma \ref{lemmaps2} (the fact that now $B$ is not necessarily projective does not affect the arguments), and the third follows immediately.
The inequalities (\ref{sigmainequalities2}) then follow from (\ref{stupidbounds2}).

Local higher order estimates for the equation (\ref{eqnvcy}) imply that the bisectional curvature of $\omega_B$ is bounded by $C \sigma^{-\lambda}$ on $B\backslash S'$, up to increasing $\lambda$.  The Schwarz Lemma calculation \cite[Lemma 3.1]{To}  gives the first inequality  (\ref{ps4}), and the second follows from (\ref{ps3}).
\end{proof}

\subsection{Collapsing estimates for Ricci-flat metrics}

In this section, we give the proof of Theorem \ref{mainthm3}.  Recall that the function $\mathcal{F}$ is given by (\ref{defnF}) and depends on the two constants $A$ and $\lambda$.

\begin{lemma}\label{zero}
There are constants $A,\lambda>0$ and a positive function $h(t)\to 0$ as $t\to\infty$, such that
\begin{equation}\label{conv}
\sup_{X}\mathcal{F} |\vp-v|\leq h(t).
\end{equation}
\end{lemma}
\begin{proof} We choose $A,\lambda$ as in Lemma \ref{lemmape3}.(1).
Taking the trace with respect to $g_X$ of $\omega = \tilde{\omega}+\ddbar \varphi$ and applying Lemma \ref{lemmape3}.(i), we have
\begin{equation} \label{lapb}
\sup_{X}\mathcal{F}|\Delta_{g_X}\vp|\leq C.
\end{equation}
From Lemma \ref{lemmape3}.(ii) and the fact that $v$ is bounded we have
\begin{equation}\label{c0}
\sup_X|\vp|+\sup_X|v|\leq C.
\end{equation}
It follows that
\begin{equation}\label{bdphi}
\sup_{X} \mathcal{F}  | \de  \vp  |_{g_X}\leq C,
\end{equation}
from a rather standard interpolation type argument.  Indeed, note that the Sobolev imbedding theorem and the $L^p$ elliptic estimates give for $p> 2(n+m)$,
$$\sup_X |\de (\mathcal{F} \varphi)|_{g_X} \le C \left( \left( \int_X | \Delta_{g_X} (\mathcal{F} \varphi) |^p  \right)^{1/p} + \left( \int_X |\mathcal{F} \varphi|^p \right)^{1/p} \right),$$
where we are integrating with respect to the volume form of $g_X$.  Using (\ref{lapb}),  (\ref{c0}) and the Cauchy-Schwarz inequality, we obtain
\begin{equation} \label{numerouno}
\begin{split}
\sup_X\mathcal{F}  |\de  \vp|_{g_X} \le {} & C \left( \left( \int_X (|\partial \varphi|_{g_X} | \partial \mathcal{F}|_{g_X})^p \right)^{1/p} +1 \right) \\
 \le {} & C \left( \sup_X ( \mathcal{F} | \partial \varphi|_{g_X})^{(p-2)/p}   \left( \int_X |\partial \varphi|_{g_X}^2  \hat{\mathcal{F}} \right)^{1/p} +1 \right),
\end{split}
\end{equation}
where $\hat{\mathcal{F}} = | \partial \mathcal{F}|^p_{g_X}/\mathcal{F}^{p-2}$.  From the definition of $\mathcal{F}$ we see that $\hat{\mathcal{F}}+ | \partial \hat{\mathcal{F}}|_{g_X} \le C \mathcal{F}$  and
hence, integrating by parts,
\begin{equation} \label{numerodos}
\begin{split}
\int_X | \partial \vp|^2_{g_X} \hat{\mathcal{F}} = {} &  - \int_X  \left( \varphi ( \Delta_{g_X} \vp) \hat{\mathcal{F}}  - \varphi \langle \de \varphi, \de \hat{\mathcal{F}} \rangle_{g_X} \right)  \\ \le {} & C ( \sup_X \mathcal{F} | \partial \varphi|_{g_X} + 1).
\end{split}
\end{equation}
Then (\ref{bdphi}) follows from (\ref{numerouno}) and (\ref{numerodos}).

Next, the estimates in \cite[Section 3]{ST2}, together with a similar argument give
\begin{equation}\label{bdv}
\sup_X \mathcal{F} |\de v|_{g_X}\leq C.
\end{equation}

Combining \eqref{c0}, \eqref{bdphi} and \eqref{bdv}, we obtain
\begin{equation}\label{con1}
\sup_X |\de (\mathcal{F}(\vp-v))|_{g_X}\leq C.
\end{equation}
From Lemma \ref{lemmape3}.(iii),
\begin{equation}\label{con2}
\int_X\mathcal{F}|\vp-v|\omega_X^{n+m}\leq C\int_X|\vp-v|\omega_X^{n+m} \to 0, \quad \textrm{as } t \to \infty,
\end{equation}
and \eqref{conv} now follows  from \eqref{con1} and \eqref{con2}.
\end{proof}

Write $\dot{\varphi}$ and $\ddot{\varphi}$ for the first and second $t$-derivatives of $\varphi$.  Then:

\begin{lemma}\label{c0lem}
There is a uniform constant $C$ so that
\begin{equation}\label{bddphi}
\sup_X | \dot{\vp}|\leq C,
\end{equation}
and
\begin{equation}\label{bdddphi}
\sup_X \ddot{\vp}\leq C.
\end{equation}
\end{lemma}
\begin{proof}
First, we have
$$\Delta\vp=n+m-\tr{\omega}{\ti{\omega}}.$$
Differentiating the logarithm of \eqref{ma} with respect to $t$, we obtain
\begin{equation}\label{phidot}
(\log c_t)'-n=\tr{\omega}({\ddbar\dot{\vp}}+\chi-\ti{\omega}),
\end{equation}
and so
$$\Delta\dot{\vp}=-n+(\log c_t)'+\tr{\omega}{\ti{\omega}}-\tr{\omega}{\chi}.$$
Hence
\begin{equation} \label{vpdvp3}
\Delta(\vp+\dot{\vp})=m-\tr{\omega}{\chi}+(\log c_t)'.
\end{equation}
We have that $(\log c_t)'=O(e^{-t})$, because thanks to \eqref{ct} we can write
$$\log c_t=\log\left(1+\sum_{k=1}^m a_ke^{-kt}\right)+{\rm const},$$
for some positive constants $a_k$, and the claim follows from differentiating this expression. Taking one more derivative, we also see that $(\log c_t)''=O(e^{-t}).$

From the Schwarz Lemma (\ref{ps3}) we have $\sup_X \tr{\omega}{\chi}\leq C$, and hence $|\Delta(\vp+\dot{\vp})|\leq C$. The normalization for $\vp$ implies that $\int_X \dot{\vp}\omega_X^{n+m}=0$. We now compute the Laplacian of $\ddot{\vp}$, by taking a derivative of \eqref{phidot} to get
$$\Delta \ddot{\vp}=(\log c_t)''+\tr{\omega}{\chi}-\tr{\omega}{\ti{\omega}}+|{\ddbar\dot{\vp}}+\chi-\ti{\omega}|^2_g,$$
and so $\Delta(\dot{\vp}+\ddot{\vp})\geq -C$ and $\int_X \ddot{\vp}\omega_X^{n+m}=0$.

The Ricci-flat metrics $\omega$ have a uniform upper bound on their diameter, thanks to \cite{To2} (and independently \cite{ZT}), and have volume bounded below by $C^{-1}e^{-nt}$, by \eqref{ma}. Therefore the Green's function $G(x,y)$ of the Laplacian of $\Delta$ (normalized by $\int_{y\in X} G(x,y)\omega^{n+m}(y)=0$), satisfies
$$G(x,y)\geq -Ce^{nt},$$
for a uniform constant $C$ (see e.g. \cite[Theorem 3.2]{BM} or \cite[Chapter 3, Appendix A]{Si}). Also, again from \eqref{ma},
$$\int_X(\vp+\dot{\vp})\omega^{n+m}=c_te^{-nt}\int_X(\vp+\dot{\vp})\omega_X^{n+m}=0,$$
and so the Green's formula for the metric $\omega$ and the bound $| \Delta (\varphi + \dot{\varphi})|\le C$ imply that
\[\begin{split}
\sup_X(\vp+\dot{\vp})&=-\int_{y \in X} \Delta(\vp+\dot{\vp})(y)G(x,y)\omega^{n+m}(y)\\
&=\int_{y \in X} (-\Delta(\vp+\dot{\vp})(y))(G(x,y)+Ce^{nt})\omega^{n+m}(y)\\
&\leq Ce^{nt}\int_X \omega^{n+m}\leq C,
\end{split}\]
where $x$ is any point where $\vp+\dot{\vp}$ achieves its maximum. Similarly we get a lower bound for $\vp+\dot{\vp}$, and so we have shown
\begin{equation}\label{bddd}
\sup_X|\vp+\dot{\vp}|\leq C,
\end{equation}
which, together with $\sup_X|\vp|\leq C$ (Lemma \ref{lemmape3}.(ii)), proves \eqref{bddphi}.

Next, since $\Delta(\dot{\vp}+\ddot{\vp})\geq -C$ and $\int_X (\dot{\vp}+\ddot{\vp})\omega_X^{n+m}=0$, exactly the same argument shows that
$$\sup_X (\dot{\vp}+\ddot{\vp})\leq C,$$
which together with \eqref{bddphi} proves \eqref{bdddphi}.
\end{proof}

\begin{remark}
One can feed the bound \eqref{bddd} into a Cheng-Yau type argument, exactly as in \cite[Lemma 3.6]{So2}, by applying the maximum principle to
$$Q=\frac{|\de(\vp+\dot{\vp})|_{g}^2}{A-\vp-\dot{\vp}}+\tr{\omega}{\chi},$$
where $A$ is chosen so that $A-\vp-\dot{\vp}\geq A/2>0$, and prove that
$$\sup_X |\de(\vp+\dot{\vp})|_{g}\leq C.$$
However, we won't need this estimate.
\end{remark}

The quantity $(\varphi + \dot{\varphi}-v)$ will play here the same role as in Section \ref{sectionkrfsing}.  Indeed,
observe that from (\ref{vpdvp3}),
\begin{equation} \label{vpdvpmv}
\Delta (\varphi + \dot{\varphi} - v) = m- \tr{\omega}{\omega_B} + (\log c_t)',
\end{equation}
and $(\log c_t)' = O(e^{-t})$ (cf. (\ref{pdpv})).   The following lemma is analogous to part (iv) of Lemma \ref{lemmape2}.

\begin{lemma} \label{decaylemma}
There is a positive function $H(t)$ with $H(t)\to 0$ as $t\to\infty$ such that
\begin{equation}\label{decay}
\sup_X \mathcal{F} |\vp+\dot{\vp}-v|\leq H(t).
\end{equation}
\end{lemma}
\begin{proof}
First we prove that $\limsup_{t\geq 0}\sup_X \mathcal{F}(\vp+\dot{\vp}-v)\leq 0$.
Thanks to \eqref{conv} it is enough to show that
$$\limsup_{t\geq 0}\sup_X\mathcal{F}\dot{\vp}\leq 0.$$ If this is not the case, then there exist $\ve>0$, $x_k \in X$ and $t_k\to\infty$ such that
$$\sup_X\mathcal{F}\dot{\vp}(t_k)=\mathcal{F}(x_k)\dot{\vp}(t_k,x_k)\geq \ve,$$
so in particular $x_k\not\in S$. From Lemma \ref{c0lem} we see that
\begin{equation}\label{change}
\frac{\de}{\de t}\left( \mathcal{F}\dot{\vp}\right)= \mathcal{F}\ddot{\vp}\leq C,
\end{equation}
and so
$$\mathcal{F}(x_k) \dot{\vp}(t,x_k)\geq \frac{\ve}{2}, \quad \textrm{for } t\in [t_k-\frac{\ve}{2C}, t_k].$$
Integrating over $t$,
$$\mathcal{F}(x_k)(\vp-v)(t_k,x_k)\geq \mathcal{F}(x_k)(\vp-v)(t_k-\ve/2C,x_k)+\frac{\ve^2}{4C},$$
but from \eqref{conv} we have
$$h(t_k)\geq -h(t_k-\ve/2C)+\frac{\ve^2}{4C},$$
for a positive function $h(t)$ with $h(t) \rightarrow 0$ as $t \rightarrow \infty$.  Letting $k\to\infty$ we get a contradiction.

The inequality $\liminf_{t\geq 0}\inf_X \F(\vp+\dot{\vp}-v)\geq 0$ follows from a similar argument.   Although we do not have a lower bound for $\ddot{\varphi}$, we can instead replace the time interval $[t_k-\frac{\ve}{2C}, t_k]$ by $[t_k, t_k+ \frac{\ve}{2C}]$.
\end{proof}

We can finally prove the following analogue of Lemma \ref{con22}:
\begin{lemma}\label{con}
Given a compact set $K\subset X\backslash S$, there is a positive decreasing function $F(t)$ which goes to zero as $t \rightarrow \infty$ such that
 $$\sup_K(\tr{\omega}{\omega_B}-m) \le F(t).$$
\end{lemma}
\begin{proof}
For the purpose of this proof,  define a function $\tilde{\F} = e^{-e^{\tilde{A}\sigma^{-\lambda}}}$, with $\tilde{A}$ slightly larger than the constant $A$ of $\mathcal{F}$ used in (\ref{decay}).
We will apply the maximum principle to
$$Q=\tilde{\mathcal{F}} (\tr{\omega}{\omega_B}-m)-\frac{\tilde{\mathcal{F}}}{\sqrt{H(t)}}(\vp+\dot{\vp}-v),$$
where $H(t)$ is the function from \eqref{decay}.

We may assume, by increasing $\tilde{A}$ if necessary, that
\begin{equation}
\frac{|\nabla \tilde{\F}|^2}{\tilde{\mathcal{F}}} \leq C \mathcal{F}, \quad | \Delta \tilde{\mathcal{F}}|\le C \mathcal{F},
\end{equation}
where we used \eqref{stupidbounds2} and the Schwarz Lemma estimate $\tr{\omega}{\chi}\leq C$. It follows from Lemma \ref{decaylemma} that
\begin{equation} \label{vpdf}
|(\vp+\dot{\vp}-v)\Delta \tilde{\F}|\leq CH(t)
\end{equation}
and
\begin{equation}\label{idiot}
|\tilde{\F}(\vp+\dot{\vp}-v)| \frac{|\nabla \tilde{\F}|^2}{\tilde{\F}^2}\leq CH(t).
\end{equation}

Then we compute, using (\ref{vpdvpmv}) and (\ref{vpdf}),
\[\begin{split}
\lefteqn{\Delta(\tilde{\F}(\vp+\dot{\vp}-v))} \\  \le {} & -\tilde{\F}(\tr{\omega}{\omega_B}-m-(\log c_t)')  +2\mathrm{Re}\langle\nabla \tilde{\F}, \nabla(\vp+\dot{\vp}-v)\rangle + CH(t). \\
\end{split}\]
From Lemma \ref{lemmaps3}, we have, for some $\lambda'>0$,
\[\begin{split}
\Delta(\tilde{\F}(\tr{\omega}{\omega_B}-m))&\geq-C\tilde{\F} \sigma^{-\lambda'}+(\tr{\omega}{\omega_B}-m)\Delta\tilde{\F}\\
&+2\mathrm{Re}\langle\nabla \tilde{\F}, \nabla(\tr{\omega}{\omega_B}-m)\rangle\\
&\geq 2\mathrm{Re}\langle\nabla \tilde{\F}, \nabla(\tr{\omega}{\omega_B}-m)\rangle-C.
\end{split}\]
Hence
\[\begin{split}
\Delta Q&\geq\frac{\tilde{\F}}{\sqrt{H(t)}}(\tr{\omega}{\omega_B}-m-(\log c_t)')-C
+2\mathrm{Re}\langle\nabla \tilde{\F}, \nabla(Q/\tilde{\F})\rangle\\
&=\frac{\tilde{\F}}{\sqrt{H(t)}}(\tr{\omega}{\omega_B}-m-(\log c_t)')-C+\frac{2}{\tilde{\F}}\mathrm{Re}\langle\nabla \tilde{\F},\nabla Q\rangle  -\frac{2Q}{\tilde{\F}^2}|\nabla\tilde{\F}|^2\\
&\geq \frac{\tilde{\F}}{\sqrt{H(t)}}(\tr{\omega}{\omega_B}-m-(\log c_t)')-C+\frac{2}{\tilde{\F}}\mathrm{Re}\langle\nabla \tilde{\F},\nabla Q\rangle,
\end{split}\]
where in the last line we used \eqref{idiot}. At the maximum of $Q$ we have $\nabla Q=0$.
 We may assume without loss of generality that $(\log c_t)'/\sqrt{H(t)} \rightarrow 0$ as $t \rightarrow \infty$.   Making use of Lemma \ref{decaylemma}, we obtain $Q \le C \sqrt{H(t)}$.  The result follows.
\end{proof}

Another ingredient that we will need is the following local Calabi estimate, whose proof can be obtained by a straightforward modification of the proof of Proposition \ref{1localest}.
\begin{proposition}\label{localest3}
Let $B_1(0)$ be the unit polydisc in $\mathbb{C}^{n+m}$ and let $\omega_E^{(n+m)} = \sum_{k=1}^{m+n} \sqrt{-1} dz^k \wedge d\ov{z}^k$ be the Euclidean metric.
Write $\mathbb{C}^{n+m}=\mathbb{C}^{m}\oplus\mathbb{C}^{n}$ and define $\omega_{E,t}=\omega_{E}^{(m)} + e^{-t}\omega^{(n)}_E$.
Assume that $\omega=\omega(t)$ is a Ricci-flat K\"ahler metric which satisfies for $t \ge 0$,
\begin{equation}\label{2assum}
A^{-1}\omega_{E,t}\leq \omega \leq A\, \omega_{E,t},
\end{equation}
for some positive constant $A$.
Then there is a constant $C$ that depends only on $n,m,A$ such that for all $t\geq 0$ on $B_{1/2}(0)$ we have
\begin{equation}\label{2bo}
S=|\nabla^E g|^2_{g} \leq Ce^t,
\end{equation}
where $g$ is the metric associated to $\omega$ and $\nabla^E$ is the covariant derivative of $\omega_E^{(n+m)}$.
\end{proposition}

We can now complete the proof of Theorem \ref{mainthm3}.

\begin{proof}[Proof of Theorem \ref{mainthm3}]
The proof is similar to the proof of Theorem \ref{mainthm2}, so we provide here just a sketch.   Fix a compact set $K' \subset B\backslash S$ and write $K=\pi^{-1}(K')$.
From Lemma \ref{lemmape3}.(i) and Proposition \ref{localest3}, as in the case of the flow, there is a constant $C$ such that
\begin{equation}\label{c1est}
\|e^t\omega|_{X_y}\|_{C^1(X_y,\omega_X|_{X_y})}\leq C , \quad e^t\omega|_{X_y}\geq C^{-1}\omega_X|_{X_y},
\end{equation}
for all $t\geq 0$ and for all $y\in K'$.

On $X\backslash S$, write for any $y \in K'$,
\[\begin{split}
(e^t\omega|_{X_y})^n&= e^{nt}\frac{\omega^n\wedge\omega_B^{m}}{\omega_{{\rm SRF}}^n\wedge\omega_B^m}(\omega_{{\rm SRF},y})^n\\
&=f (\omega_{{\rm SRF},y})^n, \quad \textrm{for } f = c_t\frac{\int_{X}\omega_X^{n+m}}{\int_{X} \omega_X^n\wedge\chi^m}\frac{\omega^n\wedge\omega_B^{m}}{\omega^{n+m}},
\end{split}\]
where we used \eqref{ct3} and \eqref{ma}.  Note that when restricted to the fiber, $f$ is given by
\begin{equation}\label{fib}
f|_{X_y}=\frac{(e^t\omega|_{X_y})^n}{(\omega_{{\rm SRF},y})^n}.
\end{equation}
We have that
\begin{equation}\label{intt}
\int_{X_y}f(\omega_{\textrm{SRF},y})^n=\int_{X_y} (e^t\omega|_{X_y})^n=\int_{X_y}(\omega_X|_{X_y})^n=\int_{X_y}(\omega_{\textrm{SRF},y})^n,
\end{equation}
so that $f-1$ satisfies condition (b) of Lemma \ref{1conve}. Thanks to \eqref{ct2},
\begin{equation}\label{expd}
c_t\frac{\int_{X}\omega_X^{n+m}}{\int_{X} \omega_X^n\wedge\chi^m}\to\binom{n+m}{n},
\end{equation}
as $t\to\infty$.
Moreover, as in (\ref{1amgm}), using Lemma \ref{con},
$$\binom{n+m}{n}\frac{\omega^n\wedge\omega_B^{m}}{\omega^{n+m}}\le (\tr{\omega}{\omega_B}/m)^{m} \leq 1+ H(t),$$
for $H(t) \rightarrow 0$ as $t\rightarrow \infty$.  Hence $f-1 \le h(t),$ with $h(t)\rightarrow 0$ as $t \rightarrow \infty$.  Finally, $f$ satisfies
$| \nabla (f|_{X_y})|_{\omega_X|_{X_y}} \le A$ for all $y \in K'$.  Applying Lemma \ref{1conve} to $f-1$ on $K$ (see Remark \ref{remarkcpt}), $f$ converges to $1$ uniformly on $K$.  Namely,
\begin{equation}\label{volcon}
\|(e^{t}\omega|_{X_y})^n-\omega_{{\rm SRF},y}^n\|_{C^0(X_y,\omega_X|_{X_y})}\to 0, \quad \textrm{as } t\rightarrow \infty,
\end{equation}
uniformly as $y$ varies in $K'$.  It follows that for any given $y \in K'$ we have
$$e^t \omega|_{X_y} \rightarrow \omega_{\textrm{SRF},y}, \quad \textrm{as } t\rightarrow \infty,$$
in $C^{\alpha}$ on $X_y$, for any $\alpha \in (0,1)$.

Next we show that
\begin{equation} \label{fwca}
\| e^t \omega|_{X_y} - \omega_{\textrm{SRF},y} \|_{C^0(X_y, \omega_X|_{X_y})} \rightarrow 0, \quad \textrm{as } t \rightarrow \infty,
\end{equation}
\emph{uniformly} for $y \in K'$.  To see this, define
$$f|_{X_y}=\frac{(e^t\omega|_{X_y})\wedge \omega_{{\rm SRF},y}^{n-1}}{\omega_{{\rm SRF},y}^n}.$$
While this is defined on $X_y$, it is clearly smooth in $y$ and so defines a smooth function $f$ on $X\backslash S$, which equals
\[\begin{split}f&=\frac{e^t\omega\wedge \omega_{{\rm SRF}}^{n-1}\wedge\omega_B^m}{\omega_{{\rm SRF}}^n\wedge\omega_B^m}
=c_t\frac{\int_{X}\omega_X^{n+m}}{\int_{X} \omega_X^n\wedge\chi^m}\frac{\omega\wedge (e^{-t}\omega_{{\rm SRF}})^{n-1}\wedge\omega_B^m}{\omega^{n+m}}.
\end{split}\]
By the same argument as in the proof of (\ref{mtii}) in Section 2, we see that $f$ converges 1 as $t \rightarrow 0$ and hence
\begin{equation} \label{g2}
\| (e^t \omega|_{X_y}) \wedge \omega_{\textrm{SRF},y}^{n-1} - \omega_{\textrm{SRF},y}^n \|_{C^0(X_y, \omega_X|_{X_y})} \rightarrow 0, \quad \textrm{as } t \rightarrow \infty,
\end{equation}
uniformly as $y$ varies in $K'$.  From (\ref{volcon}) and (\ref{g2}), we apply  Lemma \ref{lemmamatrix} to obtain (\ref{fwca}) as required.

It remains to prove part (i) of Theorem \ref{mainthm3}.  Define $\hat{\omega}=\hat{\omega}(t)$ by
$$\hat{\omega} = e^{-t} \omega_{\textrm{SRF}} + (1-e^{-t}) \omega_B.$$
 By the same argument as in the proof of Theorem \ref{mainthm2}.(i), we have on $K$,
\begin{equation}\label{one1}
\tr{\omega}{\hat{\omega}}-(n+m)\leq h(t) \rightarrow 0,
\end{equation}
for a positive decreasing function $h(t)$ (depending on $K$).
On the other hand, again on $K$,
\[
\begin{split}
\frac{\hat{\omega}^{n+m}}{\omega^{n+m}}= {} & \binom{n+m}{n}\frac{e^{-nt}\of^n\wedge\omega_B^m}{\omega^{n+m}}+O(e^{-t}) \\
= {} & \frac{\binom{n+m}{n}\int_X\omega_X^n\wedge\chi^m}{c_t\int_X\omega_X^{n+m}}+O(e^{-t})\to 1, \quad \textrm{as } t\rightarrow \infty.
\end{split}
\]
Applying Lemma \ref{lemmamatrix}, we have $\| \omega - \hat{\omega}\|_{C^0(K, \omega)}   \rightarrow 0$  and part (i) of Theorem \ref{mainthm3} follows.
\end{proof}


\begin{thebibliography}{99}
\bibitem{Au} Aubin, T. {\em \'Equations du type Monge-Amp\`ere sur les vari\'et\'es k\"ahleriennes compactes}, C. R. Acad. Sci. Paris S\'er. A-B {\bf 283} (1976), no. 3, Aiii, A119--A121.
\bibitem{BM} Bando, S., Mabuchi, T. {\em Uniqueness of Einstein K\"ahler metrics modulo connected group actions} in {\em Algebraic geometry, Sendai, 1985,}  11--40, Adv. Stud. Pure Math., 10, North-Holland, Amsterdam, 1987.
\bibitem{bhpv} Barth, W.P., Hulek, K., Peters, C.A.M., Van de Ven, A. {\em Compact complex surfaces. Second edition}, Springer, 2004.
\bibitem{BT} Bedford, E., Taylor, B.A. \emph{The Dirichlet problem for a complex Monge-Amp\`ere equation}, Invent. Math. {\bf 37} (1976), no. 1, 1--44.
\bibitem{BCHM} Birkar, C, Cascini, P., Hacon, C. D., McKernan, J. {\em Existence of minimal models for varieties of log general type}, J. Amer. Math. Soc. {\bf 23} (2010), no. 2, 405--468.
\bibitem{Ca} Cao, H.-D. {\em Deformation of K\"ahler metrics to K\"ahler-Einstein metrics on compact K\"ahler manifolds}, Invent. Math. {\bf 81}  (1985), no. 2, 359--372.
\bibitem{CC}  Cheeger, J., Colding, T.H. {\em On the structure of spaces with Ricci curvature bounded below. I}, J. Differential Geom. {\bf 46}  (1997),  no. 3, 406--480.
\bibitem{CT} Collins, T.C., Tosatti, V. {\em K\"ahler currents and null loci}, Invent. Math. {\bf 202} (2015), no. 3, 1167--1198.
\bibitem{EGZ} Eyssidieux, P., Guedj, V., Zeriahi, A. \emph{Singular K\"ahler-Einstein metrics}, J. Amer. Math. Soc. {\bf 22}  (2009), 607--639.
\bibitem{EGZ2} Eyssidieux, P., Guedj, V., Zeriahi, A. {\em Weak solutions to degenerate complex Monge-Amp\`ere flows II},  Adv. Math. {\bf 293} (2016), 37--80.
\bibitem{Fi} Fine, J. {\em Fibrations with constant scalar curvature K\"ahler metrics and the CM-line bundle}, Math. Res. Lett. {\bf 14} (2007), no. 2, 239--247.
\bibitem{FZ}  Fong, F. T.-H.,  Zhang, Z. \emph{The collapsing rate of the K\"ahler-Ricci flow with regular infinite time singularity}, J. reine angew. Math. {\bf 703} (2015), 95--113.
\bibitem{FS} Fujiki, A., Schumacher, G. {\em The moduli space of extremal compact K\"ahler manifolds and generalized Weil-Petersson metrics}, Publ. Res. Inst. Math. Sci. {\bf 26} (1990), no. 1, 101--183.
 \bibitem{Gi0}  Gill, M. {\em Convergence of the parabolic complex Monge-Amp\`ere equation on compact Hermitian manifolds}, Comm. Anal. Geom. {\bf 19} (2011), no. 2, 277--303.
\bibitem{Gi} Gill, M. {\em Collapsing of products along the K\"ahler-Ricci flow}, Trans. Amer. Math. Soc. {\bf 366} (2014), no. 7, 3907--3924.
\bibitem{GSVY} Greene, B., Shapere, A., Vafa, C., Yau, S.-T. {\em Stringy cosmic strings and noncompact Calabi-Yau manifolds}, Nuclear Phys. B {\bf 337} (1990), no. 1, 1--36.
\bibitem{Gr}  Griffiths, P.A. {\em Periods of integrals on algebraic manifolds. II. Local study of the period mapping}, Amer. J. Math. {\bf 90} (1968) 805--865.
\bibitem{GTZ} Gross, M., Tosatti, V., Zhang, Y. {\em Collapsing of abelian fibred Calabi-Yau manifolds}, Duke Math. J. {\bf 162} (2013), no. 3, 517--551.
\bibitem{GTZ2} Gross, M., Tosatti, V., Zhang, Y. {\em Gromov-Hausdorff collapsing of Calabi-Yau manifolds}, Comm. Anal. Geom. {\bf 24} (2016), no. 1, 93--113.
\bibitem{GW} Gross, M., Wilson, P.M.H. \emph{Large complex structure limits of $K3$ surfaces}, J. Differential Geom. {\bf 55} (2000), no. 3, 475--546.
\bibitem{Ha} Hamilton, R.S. {\em Three-manifolds with positive Ricci curvature}, J. Differential Geom. {\bf 17} (1982), no. 2, 255--306.
\bibitem{HT} Hein, H.-J., Tosatti, V. {\em Remarks on the collapsing of torus fibered Calabi-Yau manifolds}, Bull. Lond. Math. Soc. {\bf 47} (2015), no. 6, 1021--1027.
\bibitem{Ko} Kodaira, K. {\em Complex manifolds and deformation of complex structures}, Springer, 2005.
\bibitem{Kol} Ko{\l}odziej, S. {\em The complex Monge-Amp\`ere equation}, Acta Math. {\bf 180} (1998), no. 1, 69--117.
\bibitem{KS} Kontsevich, M., Soibelman, Y. {\em  Homological mirror symmetry and torus fibrations}, in {\em Symplectic geometry and mirror symmetry},  203--263, World Sci. Publishing 2001.
\bibitem{LT} La Nave, G., Tian, G. {\em Soliton-type metrics and K\"ahler-Ricci flow on symplectic quotients}, J. reine angew. Math. {\bf 711} (2016), 139--166.
\bibitem{La} Lazarsfeld, R. {\em Positivity in algebraic geometry. I \& II}, Springer, 2004.
\bibitem{Ma} Manin, Y.I. {\em  Moduli, motives, mirrors}, in {\em European Congress of Mathematics, Vol. I (Barcelona, 2000)}, 53--73, Progr. Math., 201, Birkh\"auser, Basel, 2001.
\bibitem{PSS} Phong,  D.H., Sesum, N., Sturm, J. {\em Multiplier Ideal Sheaves and the K\"ahler-Ricci Flow},  Comm. Anal. Geom.  {\bf 15}  (2007),  no. 3, 613--632.
\bibitem{PSoS} Phong, D. H., Song, J., Sturm, J. {\em Complex Monge-Amp\`ere equations}, in {\em Surveys in differential geometry}, Vol. XVII, 327--410, Surv. Differ. Geom., 17, Int. Press, Boston, MA, 2012.
\bibitem{RZ} Rong, X., Zhang, Y. {\em Continuity of extremal transitions and flops for Calabi-Yau manifolds}, J.  Differential Geom. {\bf 89} (2011), no. 2, 233--269.
\bibitem{RZ2} Rong, X., Zhang, Y. {\em Degenerations of Ricci-flat Calabi-Yau manifolds},  Commun. Contemp. Math. {\bf 15} (2013), no. 4, 1250057, 8 pp.
\bibitem{RZ3} Ruan, W., Zhang, Y. {\em Convergence of Calabi-Yau manifolds}, Adv. Math. {\bf 228} (2011), no. 3, 1543--1589.
\bibitem{ShW} Sherman, M., Weinkove, B. {\em Interior derivative estimates for the K\"ahler-Ricci flow}, Pacific J. Math. {\bf 257} (2012), no. 2, 491--501.
\bibitem{Si} Siu, Y.-T. {\em Lectures on Hermitian-Einstein metrics for stable bundles and K\"ahler-Einstein metrics}, DMV Seminar, 8. Birkh\"auser Verlag, Basel, 1987.
\bibitem{So} Song, J. {\em Ricci flow and birational surgery}, preprint, arXiv:1304.2607.
\bibitem{So2} Song, J. {\em Riemannian geometry of K\"ahler-Einstein currents}, preprint, arXiv:1404.0445.
\bibitem{SSW} Song, J., Sz\'ekelyhidi, G., Weinkove, B. {\em The K\"ahler-Ricci flow on projective bundles}, Int. Math. Res. Not. IMRN 2013, no. 2, 243--257.
\bibitem{ST} Song, J., Tian, G. {\em The K\"ahler-Ricci flow on surfaces of positive Kodaira dimension}, Invent. Math. {\bf 170} (2007), no. 3, 609--653.
\bibitem{ST2} Song, J., Tian, G. {\em Canonical measures and K\"ahler-Ricci flow}, J. Amer. Math. Soc. {\bf 25} (2012), no. 2, 303--353.
\bibitem{ST4} Song, J., Tian, G. {\em The K\"ahler-Ricci flow through singularities}, to appear in Invent. Math.
\bibitem{ST3} Song, J., Tian, G. {\em Bounding scalar curvature for global solutions of the K\"ahler-Ricci flow}, Amer. J. Math. {\bf 138} (2016), no. 3, 683--695.
\bibitem{SW0} Song, J., Weinkove, B. {\em The K\"ahler-Ricci flow on Hirzebruch surfaces}, J. Reine Angew. Math. {\bf 659} (2011), 141--168.
\bibitem{SW1} Song, J., Weinkove, B. {\em Contracting exceptional divisors by the K\"ahler-Ricci flow}, Duke Math. J. {\bf 162} (2013), no. 2, 367--415.
\bibitem{SW} Song, J., Weinkove, B. {\em Introduction to the K\"ahler-Ricci flow},  Chapter 3 of `Introduction to the K\"ahler-Ricci flow', eds S. Boucksom, P. Eyssidieux, V. Guedj, Lecture Notes Math. 2086, Springer 2013.
\bibitem{SW2} Song, J., Weinkove, B. {\em Contracting exceptional divisors by the K\"ahler-Ricci flow, II},  Proc. Lond. Math. Soc. (3) {\bf 108} (2014), no. 6, 1529--1561.
\bibitem{SY} Song, J., Yuan, Y. {\em Metric flips with Calabi ansatz},  Geom. Funct. Anal. {\bf 22} (2012), no. 1, 240--265.
\bibitem{SYZ} Strominger, A., Yau, S.-T., Zaslow, E. {\em Mirror symmetry is $T$-duality}, Nuclear Phys. B {\bf 479} (1996), no. 1-2, 243--259.
\bibitem{Ti} Tian, G. {\em Smoothness of the universal deformation space of compact Calabi-Yau manifolds and its Petersson-Weil metric}, in {\em Mathematical aspects of string theory (San Diego, Calif., 1986)}, 629--646, Adv. Ser. Math. Phys., 1, World Sci. Publishing, Singapore, 1987.
\bibitem{Ti2} Tian, G. {\em New results and problems on K\"ahler-Ricci flow}, Ast\'erisque No. {\bf 322} (2008), 71--92.
\bibitem{TZ} Tian, G., Zhang, Z. {\em On the K\"ahler-Ricci flow on projective manifolds of general type},  Chinese Ann. Math. Ser. B  {\bf 27}  (2006),  no. 2, 179--192.
\bibitem{Tod}  Todorov, A.N. {\em The Weil-Petersson geometry of the moduli space of $SU(n\geq 3)$ (Calabi-Yau) manifolds. I}, Comm. Math. Phys. {\bf 126} (1989), no. 2, 325--346.
\bibitem{To2} Tosatti, V. {\em Limits of Calabi-Yau metrics when the K\"{a}hler class degenerates}, J. Eur. Math. Soc. (JEMS) {\bf 11} (2009), no.4, 755--776.
\bibitem{To} Tosatti, V. {\em Adiabatic limits of Ricci-flat K\"ahler metrics}, J. Differential Geom. {\bf 84} (2010), no.2, 427--453.
\bibitem{To3} Tosatti, V. {\em Degenerations of Calabi-Yau metrics}, in {\em Geometry and Physics in Cracow,} Acta Phys. Polon. B Proc. Suppl. {\bf 4} (2011), no.3, 495--505.
\bibitem{To4} Tosatti, V. {\em Calabi-Yau manifolds and their degenerations}, Ann. N.Y. Acad. Sci. {\bf 1260} (2012), 8--13.
\bibitem{TW} Tosatti, V., Weinkove, B. {\em On the evolution of a Hermitian metric by its Chern-Ricci form},  J. Differential Geom. {\bf 99} (2015), no. 1, 125--163.
\bibitem{TWY} Tosatti, V., Weinkove, B., Yang, X. {\em Collapsing of the Chern-Ricci flow on elliptic surfaces},  Math. Ann. {\bf 362} (2015), no. 3-4, 1223--1271.
\bibitem{ToZ} Tosatti, V., Zhang, Y. {\em Triviality of fibered Calabi-Yau manifolds without singular fibers},  Math. Res. Lett. {\bf 21} (2014), no. 4, 905--918.
\bibitem{ToZ2} Tosatti, V., Zhang, Y. {\em Infinite time singularities of the K\"ahler-Ricci flow}, Geom. Topol. {\bf 19} (2015), no. 5, 2925--2948.
\bibitem{Ts} Tsuji, H. \emph{Existence and degeneration of K\"ahler-Einstein metrics on minimal algebraic varieties of general type}, Math. Ann. {\bf 281} (1988), no. 1, 123--133.
\bibitem{Ya} Yau, S.-T.  \emph{On the Ricci curvature of a compact K\"ahler manifold and the complex Monge-Amp\`ere equation, I}, Comm. Pure Appl. Math. {\bf 31} (1978), 339--411.
\bibitem{Ya2} Yau, S.-T. {\em A general Schwarz lemma for K\"ahler manifolds}, Amer. J. Math. {\bf 100} (1978), no. 1, 197--203.
\bibitem{ZT} Zhang, Y. {\em Convergence of K\"ahler manifolds and calibrated fibrations}, PhD thesis, Nankai Institute of Mathematics, 2006.
\bibitem{Zh} Zhang, Z. {\em  Scalar curvature bound for K\"ahler-Ricci flows over minimal manifolds of general type}, Int. Math. Res. Not. IMRN 2009, no. 20, 3901--3912.
\end{thebibliography}
\end{document}